\documentclass[11pt]{article}
\usepackage{amsmath,amsthm,amssymb,bm}
\usepackage{amsfonts}
\usepackage{mathrsfs}
\usepackage{graphicx}
\usepackage{color}
\usepackage{amscd}
\usepackage{enumitem}

\newtheorem{theorem}{Theorem}[section]
\newtheorem{lemma}[theorem]{Lemma}

\newtheorem{proposition}[theorem]{Proposition}

\theoremstyle{remark}
\newtheorem{remark}[theorem]{Remark}
\numberwithin{equation}{section}
\newcommand{\eps}{\varepsilon}
\newcommand{\R}{\mathbb{R}}
\newcommand{\F}{\mathcal{F}}
\newcommand{\Pb}{\mathbb{P}}
\newcommand{\E}{\mathbb{E}}
\newcommand{\G}{\mathcal{G}}

\newcommand{\A}{\mathcal{A}}

\def\HMM{\mathrm{HMS}}
\newcommand{\aint}{{\int \hspace{-1.1em}-\hspace{-0.2em}}}

\newcommand{\Iinner}{\mathcal{I}}
\newcommand{\Icomp}{\overline{\mathcal{I}}}
\newcommand{\dind}{\mathrm{d}}
\newcommand{\K}{\mathcal{K}}

\newcommand{\Var}{\mathrm{Var}}

\title{Corrector Analysis of a Heterogeneous Multi-scale Scheme for Elliptic Equations with Random Potential}
\author{Guillaume Bal\thanks{Department of Applied 
Physics and Applied Mathematics, Columbia University, New York, NY 10027. Email: {\tt gb2030@columbia.edu}}
\and Wenjia Jing\thanks{D\'epartement de Math\'ematiques et Applications, Ecole Normale Sup\'erieure, 45 Rue d'Ulm, 75230 Paris Cedex 05, France. Email: {\tt wjing@dma.ens.fr}}}

\date{June 3, 2013}
\begin{document}
\maketitle

\begin{abstract}

This paper analyzes the random fluctuations obtained by a heterogeneous multi-scale first-order finite element method applied to solve elliptic equations with a random potential. Several multi-scale numerical algorithms have been shown to correctly capture the homogenized limit of solutions of elliptic equations with coefficients modeled as stationary and ergodic random fields. Because theoretical results are available in the continuum setting for such equations, we consider here the case of a second-order elliptic equations with random potential in two dimensions of space.

We show that the random fluctuations of such solutions are correctly estimated by the heterogeneous multi-scale algorithm when appropriate fine-scale problems are solved on subsets that cover the whole computational domain. However, when the fine-scale problems are solved over patches that do not cover the entire domain, the random fluctuations may or may not be estimated accurately. In the case of random potentials with short-range interactions, the variance of the random fluctuations is amplified as the inverse of the fraction of the medium covered by the patches. In the case of random potentials with long-range interactions, however, such an amplification does not occur and random fluctuations are correctly captured independent of the (macroscopic) size of the patches. 

These results are consistent with those obtained in \cite{BJ-MMS-11} for more general equations in the one-dimensional setting and provide indications on the loss in accuracy that results from using coarser, and hence computationally less intensive, algorithms.

\end{abstract}

\vskip 1em


\vskip 1em


\section{Introduction}
\label{sec:intro}

Differential equations with highly oscillatory coefficients arise naturally in many areas of applied sciences. The microscopic details of such equations are difficult to compute. Nevertheless, when the heterogeneous medium has certain properties involving separation of scales, periodicity, or stationary ergodicity, homogenization theories have been developed and they provide macroscopic models for the heterogeneous equations; see e.g. \cite{JKO-SV-94,Kozlov,PV-81}. Many multi-scale algorithms have been devised to capture as much of the microscopic scale as possible without solving all the details of the micro-structure \cite{AB-MMS-05,A-SINUM-04,EMZ-JAMS-05,EE-HMM-REV,HWC-MC-99}. Such a scheme is viewed as correct if it can well approximate the macroscopic solution when the heterogeneous medium satisfies conditions for homogenization to happen. Homogenization theory thus serves as a benchmark which ensures that the multi-scale scheme performs well in controlled environments, with the hope that it will still perform well in non-controlled environments, for instance when ergodicity and stationarity assumptions are not valid.

In many applications such as parameter estimation and uncertainty quantification, estimating the random fluctuations (finding the random corrector) in the solution is as important as finding its homogenized limit \cite{BR-IP-09,NP-IP-09}. When this is relevant, another benchmark for multi-scale numerical schemes that addresses the limiting stochasticity of the solutions is plausible: One computes the limiting (probability) distribution of the random fluctuation given by the multi-scale algorithm in the limit that the correlation length of the medium tends to $0$ while the discretization size $h$ of the scheme is fixed. If this $h$-dependent distribution converges, as $h\to 0$, to the limiting distribution of the corrector of the continuous equation (before discretization), we deduce that the multi-scale algorithm asymptotically correctly captures the randomness in the solution and passes the random corrector test.

Such proposal requires a controlled environment in which the theory of correctors is available. We introduced and analyzed such a benchmark in \cite{BJ-MMS-11} using an ODE model whose corrector theory was studied in \cite{BP-99,BGMP-AA-08}. The main purpose of this paper is to provide and analyze another benchmark using a PDE model whose corrector theory was studied in \cite{B-CLH-08,FOP-SIAP-82,BGGJ-12}, hence to generalize the main results of \cite{BJ-MMS-11} in higher dimensional spaces. In the rest of this introduction, we first review some main results in \cite{BJ-MMS-11}. Then we introduce the results of the current paper that address the corrector test using an elliptic PDE with random potential.

\vskip 1em

\subsection{Corrector test using an ODE with random elliptic coefficient}

The corrector test is based on the homogenization and corrector theory of the following equation: 
\begin{equation}
\left\{
\begin{aligned}
& - \frac{d}{dx} a(\frac{x}{\eps},\omega) \frac{d}{dx} u_\eps (x, \omega) = 
f(x), \, x \in (0, 1),\\
& u_\eps(0, \omega) = u_\eps (1, \omega) = 0.
\end{aligned}
\right.
\label{eq:rode}
\end{equation}
Here, the diffusion coefficient $a(\frac{x}{\eps},\omega)$ is obtained by rescaling $a(x,\omega)$ which is a random process on some probability space $(\Omega,\F,\Pb)$. It is well known \cite{Kozlov, PV-81} that (and this generalizes to higher dimensions as well) when $a(x,\omega)$ is stationary, ergodic, and uniformly elliptic, then the solution $u_\eps$ converges to the following homogenized equation with deterministic and constant coefficient:
\begin{equation}
\left\{
\begin{aligned}
& - \frac{d}{dx} a^* \frac{d}{dx} u_0 (x) = f(x), \, x \in (0, 
1),\\
& u_0 (0) = u_0 (1) = 0.
\end{aligned}
\right.
\label{eq:hode}
\end{equation}
In the one-dimensional case, the coefficient $a^*$ is the harmonic mean of $a(x,\omega)$, i.e., the inverse of the expectation of $a^{-1}$. We denote by $q(x)$ the deviation of $1/a(x)$ from its mean $1/a^*$. The corrector theories for the limiting distribution of $u_\eps - u_0$ were studied by \cite{BGMP-AA-08,BP-99}. The results in these papers are represented in path ({\it iii}) of the diagram in Fig.~\ref{fig:cdODE}. The limiting distribution showing at the lower-right corner depends on the de-correlation rate of $q(x)$. When $q$ is strongly mixing with integrable mixing coefficient (see \eqref{eq:alphadef} below), then $\beta = 1$ and $W^\beta$ is a  standard Brownian motion multiplied by $\sigma$, a factor determined by the correlation function of $q$ as detailed in \eqref{eq:sigdef} below. When $q$ has a heavy tail (is long-range) in the sense of (L1-L3) in section \ref{sec:continuous}, we should take $\beta = \alpha$, $\alpha<1$ being defined in \eqref{eq:Rgdef}, and $W^\alpha$ is the fractional Brownian motion with Hurst index $1- {\frac \alpha 2}$ multiplied by certain factor. These convergence results are understood as convergence in distribution in the space of continuous paths $\mathcal{C}([0,1])$.

The corrector test for multi-scale numerical schemes is therefore the following: Let $h$ be the discretization size and $u^h_\eps(x)$ the solution to \eqref{eq:rode} yielded by the scheme. Let $u^h_0(x)$ be the solution yielded by the same scheme applied to \eqref{eq:hode}. The discrete corrector is $u^h_\eps - u^h_0$. According to the de-correlation property of $q(x)$, we choose $\eps^\beta$ and interpret $W^\beta$ as before. We say that a numerical procedure is consistent with the corrector theory and that it passes the corrector test when the diagram in Fig.~\ref{fig:cdODE} commutes:
\begin{figure}[h]
\caption{\label{fig:cdODE} \it A diagram describing the corrector test with a random ODE.}
\begin{equation*}
\begin{CD}
   {\displaystyle \frac{u^h_\eps - u^h_0}{\sqrt{\eps^\beta}}(x,\omega)}
   @>{h \to 0}>(i)> {\displaystyle \frac{u_\eps - 
   u_0}{\sqrt{\eps^\beta}}(x,\omega)}\\
   @V {\eps \to 0}V(ii)V  @V(iii)V{\eps \to 0}V\\
   \displaystyle \int L^h(x,y) dW^\beta(y) @>{h \to 0}>(iv)> 
   \displaystyle \int (a^*)^2 \frac{\partial G}{\partial y}(x,y) \frac{\partial u}{\partial y}(y) dW^\beta(y).
 \end{CD}
\end{equation*}
\end{figure}

More precisely, we need to characterize the intermediate limit in path ($ii$) which appears on the left of the diagram. In this step, $h$ is fixed while the correlation length $\eps$ is sent to zero. The intermediate limit distribution is $h$-dependent. Very often, it can be described as a stochastic integral as shown and we need to determine the kernel function $L^h(x,y)$. Next, we need to verify the converge path ($iv$) which is taken as $h \to 0$. The numerical scheme is said to pass (or fail) the corrector test if this limit holds (or does not).

In \cite{BJ-MMS-11}, we considered a {\it Finite Element Method} (FEM) based scheme in the framework of {\it Heterogeneous Multiscale Methods} (HMM), which is a general methodology for designing sublinear algorithms for multi-scale problems by exploiting special features of the problem, e.g. scale separation \cite{EMZ-JAMS-05}. The macro-solver of this FEM-HMM scheme uses the standard P1 element on a uniform grid of size $h$. The corresponding discrete bilinear form which approximates the continuous bilinear form associated to \eqref{eq:rode} is
\begin{equation}
\mathcal{A}^h(u^h,v^h) = \sum_{j=1}^N \frac{du^h}{dx}(x_j) a^* \frac{dv^h}{dx}(x_j) h \approx \int_0^1 \frac{du^h}{dx}(x) a^* \frac{dv^h}{dx}(x) dx = :\mathcal{A}(u^h,v^h).
\label{eq:Ahode}
\end{equation}
Here, a simple middle-point quadrature is used for the integral and $x_j$, $j=1,\cdots, N=1/h$ are the evaluation points. Since the effective coefficient $a^*$ is unknown apriori, the FEM-HMM scheme approximates the discrete integrand by 
\begin{equation*}
\frac{du^h}{dx}(x_j) a^* \frac{dv^h}{dx}(x_j) \approx \frac{1}{\delta} \int_{I_{j\delta}} \frac{d \tilde{u}^h}{dx} (x) a_\eps(x) \frac{d\tilde{v}^h}{dx} (x) dx, 
\end{equation*}
where $I_{j\delta} = (x_j-\delta/2,x_j+\delta/2)$ is a patch inside the discretization interval $I_j = (x_j-h/2,x_j+h/2)$; the functions $\tilde{u}^h$ and $\tilde{v}^h$ are given in terms of $\{\tilde{\phi}^j\}$ where $\{\phi^j\}$ are the nodal bases and $\{\tilde{\phi}^j\}$ are given by the micro-solver 
\begin{equation}
\left\{
\begin{aligned}
&-\frac{d}{dx} a_\eps(x) \frac{d}{dx} \tilde{\phi}_j(x) = 0, & & x \in I_{j\delta},\\
&\tilde{\phi}_j(x) = \phi_j(x), & &x \in \partial I_{j\delta}.
\end{aligned}
\right.
\label{eq:micro}
\end{equation}
When $\delta=h$, this scheme coincides with those in \cite{HWC-MC-99, AB-MMS-05}. It is known that one can choose $\delta < h$ to greatly reduce computational cost while still approximating the macroscopic solution quite well \cite{EMZ-JAMS-05}.

The main result of \cite{BJ-MMS-11} shows that the corrector test for the above FEM-HMM scheme depends on the correlation structure of the random media. More precisely, for a long range correlated media (L1-L3 in section \ref{sec:randf}), the scheme is robust for the corrector test: the final limit in path ({\it iv}) of the diagram in Fig.\ref{fig:cdODE} agrees with the theoretical Gaussian limit for all $\delta \le h$. For a short range correlated media (S1-S3 in section \ref{sec:randf}), however, this holds true only for $\delta=h$. The final limit for $\delta < h$ is an amplified version of the theoretical Gaussian limit with an amplification factor $(h/\delta)^{1/2}$, which shows that reducing the computational cost results in an amplification of the variance of the numerical calculations.

\subsection{Corrector test using elliptic PDE with random potential}
\label{sec:prepde}
The main objective of this paper is to provide a two dimensional corrector test. Such a strategy generalizes to arbitrary space dimensions, although for concreteness, we concentrate on the two-dimensional setting. A full theory of random fluctuations for second order elliptic PDE with highly oscillating random diffusion coefficients in dimension higher than one remains open and we can not use it for the corrector test. Instead, we base the test on the following elliptic equation with random potential:
\begin{equation}
  \left\{
  \begin{aligned}
    &-\Delta u_\eps + (q_0 + q_\eps) u_\eps(x,\omega) = f, & 
    &x \in Y,\\
    &u_\eps(x,\omega) = 0, & &x \in \partial Y.
  \end{aligned}
  \right.
  \label{eq:rpde}
\end{equation}
The coefficient in the potential term consists of a smooth varying function $q_0$, and a highly oscillatory random function $q(\eps^{-1}x,\omega)$ denoted by $q_\eps(x)$ for simplicity. The random field $q(x,\omega)$ is assumed to be stationary ergodic and mean-zero. When $\eps$ goes to zero, the solution $u_\eps$ converges in $L^2(\Omega\times Y)$ to the homogenized solution $u_0$ that 
solves
\begin{equation}
  \left\{
  \begin{aligned}
    &-\Delta u_0 + q_0 u_0(x) = f, & &x \in Y,\\
    &u_0(x) = 0, & &x \in \partial Y.
  \end{aligned}
  \right.
  \label{eq:hpde}
\end{equation}
\begin{figure}[t]
\caption{\label{fig:cdPDE} \it A diagram describing the corrector test with a random PDE.}
\begin{equation*}
\begin{CD}
   {\displaystyle \left\langle \frac{u^{h,\delta}_\eps - u^{h,\delta}_0}{\sqrt{\eps^\beta}}(x,\omega), \varphi\right\rangle}
   @>{h = \delta \to 0}>(i)> {\displaystyle \left\langle \frac{u_\eps - 
   u_0}{\sqrt{\eps^\beta}}(x,\omega), \varphi \right\rangle}\\
   @V {\substack{\eps \to 0\\h,\delta \text{ fixed}\\ \delta\le h}}V(ii)V  @V(iii)V{\eps \to 0}V\\
   \displaystyle \int_Y L^{h,\delta}[\varphi](x,y) dW^\beta(y) @>{\substack{h \to 0\\ \frac{h}{\delta} \text{ fixed}}}>(iv)> 
   \displaystyle \int_Y \varphi(x)G(x,y)u_0(y) dW^\beta(y).
 \end{CD}
\end{equation*}
\end{figure}
The corrector theory for the above homogenization is well understood; see \cite{FOP-SIAP-82,B-CLH-08,BGGJ-12}. When the corrector $u_\eps-u_0$ is properly scaled, it converges to a stochastic integral in a weak sense. This is described by the path ({\it iii}) of the diagram in Fig.~\ref{fig:cdPDE}. Both the scaling factor and the limit depend on the correlation structure of the random field. These results are reviewed in Section \ref{sec:continuous} below. As in the ODE (one-dimensional) setting, a corrector test can be sketched as in the diagram of Fig.~\ref{fig:cdPDE}. For a given multi-scale scheme, which yields $u^{h,\delta}_\eps$ and $u^{h,\delta}_0$ when it is applied to \eqref{eq:rpde} and \eqref{eq:hpde}, respectively, the main tasks are again to characterize the intermediate convergence in path ({\it ii}) where $\eps$ is sent to zero first while the parameters $h$ and $\delta$ of the scheme are fixed, and to check the validity of path ({\it iv}) where $h$ and $\delta$ are sent to zero afterwards.

\begin{figure}[hb]
\caption{\label{fig:tri} \it Left: Triangulation of the unit square. Right: Shrinking from $K$ to $K_\delta$ with respect to the barycenter.}
\setlength{\unitlength}{0.7cm}
\begin{center}
\begin{picture}(10,10)
\put(1,1){\vector(1,0){9}}
\put(1,1){\vector(0,1){9}}
\multiput(1,3)(0,2){4}{\line(1,0){8}}
\multiput(3,1)(2,0){4}{\line(0,1){8}}
\put(1,1){\line(1,1){8}}
\multiput(3,1)(-2,2){2}{\line(1,1){6}}
\multiput(5,1)(-4,4){2}{\line(1,1){4}}
\multiput(7,1)(-6,6){2}{\line(1,1){2}}
\put(9.6, 0.6){$x$}
\put(0.5,9.7){$y$}
\put(2.8,0.6){$x_1$}
\put(4.8,0.6){$x_i$}
\put(8.7,0.6){$x_N$}
\put(0.4,2.8){$y_1$}
\put(0.4,4.8){$y_j$}
\put(5,5){\circle*{0.15}}
\put(0.3,8.8){$y_N$}
\end{picture}
\begin{picture}(6,10)
\put(3,4){\line(1,0){3}}
\put(3,4){\line(0,1){3}}
\put(3,7){\line(1,-1){3}}
\put(4,5){\circle*{0.15}}
\put(3.67,4.67){\line(1,0){1}}
\put(3.67,4.67){\line(0,1){1}}
\put(3.67,5.67){\line(1,-1){1}}
\multiput(4,4.9)(0,-0.35){3}{\line(0,-1){0.2}}
\multiput(3.9,5)(-0.35,0){3}{\line(-1,0){0.2}}
\multiput(3.57,4.67)(-0.35,0){2}{\line(-1,0){0.2}}
\multiput(3.67,4.57)(0,-0.35){2}{\line(0,-1){0.2}}
\multiput(3.57,5.67)(-0.35,0){2}{\line(-1,0){0.2}}
\multiput(4.67,4.57)(0,-0.35){2}{\line(0,-1){0.2}}
\put(2.9,3.5){$0$}
\put(3.9,3.4){$\frac{h}{3}$}
\put(5.9,3.5){$h$}
\put(1.9,4.5){$\frac{h-\delta}{3}$}
\put(1.9,5.5){$\frac{h+2\delta}{3}$}
\put(2.6,6.8){$h$}
\end{picture}
\end{center}
\end{figure}
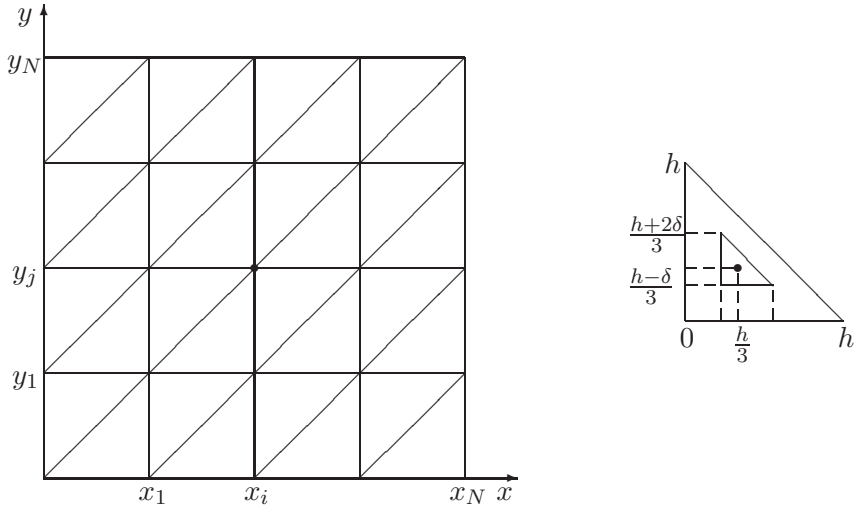

Now we introduce a heterogeneous multi-scale scheme for \eqref{eq:rpde}. The weak formulation of the equation is to find $u_\eps$ in the Sobolev space $H_0^1(Y)$ so that $\A_\eps(u_\eps, v) = \langle f, v \rangle$ for all $v \in H^1_0(Y)$. Here and below, $\langle \cdot, \cdot \rangle$ denotes the usual pairing; $\A_\eps$ is the bilinear form
\begin{equation}
\A_\eps(u,v) = \int_Y \nabla u\cdot \nabla v + (q_0 + q_\eps) uv \ dx, \quad \forall u, v \in H^1_0(Y).
\label{eq:Adef}
\end{equation}
Since we always assume that $q_0 + q_\eps$ is positive, the weak formulation is well-posed thanks to the Lax-Milgram lemma. The scheme that will be considered is based on FEM. For simplicity, $Y$ is taken as the two dimensional unit square $(0,1)^2$.  Let $\mathcal{T}_h$ be the standard uniform triangulation as illustrated in Fig.~\ref{fig:tri}. Here, the typical length of the triangles is $h=1/N$ and $N$ is the number of partitions on the axes. We consider first-order Lagrange elements. Associated to each (interior) nodal point $(ih,jh)$, there is a continuous function $\phi^{ij}$ which is linear polynomial restricted to each triangle $K \in \mathcal{T}_h$ and which has value one at this nodal point and has value zero at all other nodal points. Note that the index $i,j$ runs from $1$ to $N-1$. The space $V^h$ spanned by $\{\phi^{ij}\}$ is a finite dimensional subspace of $H_0^1(Y)$. The heterogeneous multi-scale scheme for \eqref{eq:rpde} is to find $u^{h,\delta}_\eps \in V^h$ that satisfies
\begin{equation}
\A^{h,\delta}_\eps (u^{h,\delta}_\eps, v^h) = \langle f, v^h \rangle, \quad \text{ for all } v^h \in V^h,
\label{eq:hmsdef}
\end{equation}
where $\A^{h,\delta}_\eps$ is a bilinear form on $V^h\times V^h$ which approximates $\A_\eps$ as follows:
\begin{equation}
\A^{h,\delta}_\eps (u^h, v^h) := \sum_{K \in \mathcal{T}_h} |K| \left( \frac{1}{|K_\delta|} \int_{K_\delta} \nabla u^h\cdot \nabla v^h + (q_0 + q_\eps) u^h v^h \ dx \right).
\label{eq:Ahddef}
\end{equation}
Here, $K_\delta \subset K$ is a patch centered at the barycenter of $K$ and has typical length $\delta$ (see the remark below); the symbol $|\cdot|$ means taking the area. $\A^{h,\delta}_\eps$ can be viewed as a numerical quadrature for the integral in \eqref{eq:Adef} using averaged value around the barycenters of the elements. The scheme \eqref{eq:hmsdef} is analyzed in Section \ref{sec:analysis} and it is well-posed. 

When the above scheme is applied to the homogenized equation \eqref{eq:hpde}, it yields a solution $u^{h,\delta}_0$ in $V^h$ so that
\begin{equation}
\A^{h,\delta}_0 (u^{h,\delta}_0, v^h) = \langle f, v^h \rangle, \quad \text{ for all } v^h \in V^h,
\label{eq:uh0def}
\end{equation}
and $\A^{h,\delta}_0$ is given by
\begin{equation*}
\A^{h,\delta}_0 (u^h, v^h) := \sum_{K \in \mathcal{T}_h} |K| \left( \frac{1}{|K_\delta|} \int_{K_\delta} \nabla u^h\cdot \nabla v^h + q_0 u^h v^h \ dx \right).
\end{equation*}
The discrete corrector function is defined to be the difference between $u^{h,\delta}_\eps$ and $u^{h,\delta}_0$. 
\begin{remark} \label{rem:hod} The patch $K_\delta$ is the two dimensional analog of $I_\delta$ in the aforementioned FEM-HMM scheme for the ODE setting. The ratio $|K_\delta|/|K|$ hence measures savings in the computational cost. As in the ODE setting, we expect the corrector test to depend on the ratios, say in the SRC setting. To simplify notations, we assume that $K_\delta$ is chosen in the following way: Consider a typical triangle $K$ with vertices $(0,0),(h,0)$ and $(0,h)$. $K_\delta$ is obtained by shrinking $K$ with respect to the barycenter $(h/3,h/3)$ so that it has vertices $((h-\delta)/3,(h-\delta)/3), ((h+2\delta)/3,(h-\delta)/3)$ and $((h-\delta)/3, (h+2\delta)/3)$; see Fig.~\ref{fig:tri}. Consequently we have $|K_\delta|/|K|=(\delta/h)^d$ with $d=2$. More general patches than those of the paper could also be considered without changing our main conclusions. Throughout this paper, we assume that the parameters $h$ and $\delta$, which obviously satisfies $\delta \le h$ from the above construction of $K_\delta$, are much larger than the correlation length $\eps$ of the random field so that mixing happens in the integrals of \eqref{eq:hmsdef}. Further comments on the numerical scheme can be found in section \ref{sec:fdiss} below.
\end{remark}
\subsection{Main Results}
\label{sec:main}

The main results of this paper concern the limiting distribution of the discrete corrector $u^{h,\delta}_\eps - u^{h,\delta}_0$ with proper scaling. They depend on the correlation structure of the random field $q_\eps$. We refer to section \ref{sec:randf} below for notation. In particular, SRC (respectively LRC) stands for short (respectively long) range correlation.
 
\begin{theorem} \label{thm:main} Let $u^{h,\delta}_\eps$ and $u^{h,\delta}_0$ be the solutions obtained from the heterogeneous multi-scale schemes \eqref{eq:hmsdef} and \eqref{eq:uh0def}, respectively.  Assume that $q_0 \in \mathcal{C}^1(\overline{Y})$ is positive and $f$ is in $\mathcal{C}^2(\overline{Y})$. For an arbitrary test function $\varphi \in \mathcal{C}^2(\overline{Y})$, the following holds.
\begin{enumerate}
\item[\upshape(1)] In the SRC setting, i.e.~ assume that the random field $q(x,\omega)$ satisfies {\upshape(S1)(S2)(S3)}  of section \ref{sec:randf}. Let $\sigma$ be defined by \eqref{eq:sigdef} and $L^{h,\delta}[\varphi]$ be the bounded function defined in \eqref{eq:Lhdef} below. Then as $\eps$ goes to zero while $h$ and $\delta$ with $\delta \le h$ are kept fixed, we have
\begin{equation}
\frac{1}{\sqrt{\eps^d}} \int_Y \varphi(x)[u^{h,\delta}_\eps-u^{h,\delta}_0] dx \xrightarrow[\eps \to 
0]{\mathrm{distribution}} \sigma \int_Y L^{h,\delta}[\varphi](x) d W(x),
\label{eq:main1}
\end{equation}
where $W$ is the standard multi-parameter Wiener process.
\item[\upshape(2)] Assume the same setting in {\upshape(1)}. Let $\G$ be the solution operator of \eqref{eq:hpde}. Then as $h$ and $\delta$ go to zero with the ratio $\delta/h$ being fixed, we have
\begin{equation}
\sigma \int_Y L^{h,\delta}[\varphi](x) d W(x) \xrightarrow[h \to 0]{\mathrm{distribution}} {\frac h \delta} \sigma \int_Y \G\varphi(x) u_0(x) dW(x).
\label{eq:main2}
\end{equation}

\item[\upshape(3)] In the LRC setting, i.e.~assume that $q(x,\omega)$ satisfies {\upshape(L1)(L2)(L3)} of section \ref{sec:randf}. Let $\kappa$ be defined as in the line after \eqref{eq:Hrone}.Then convergence results in item {\itshape(1)} and {\itshape(2)} are replaced by
\begin{equation}
\frac{1}{\sqrt{\eps^\alpha}} \int_Y \varphi(x)[u^{h,\delta}_\eps-u^{h,\delta}_0] dx \xrightarrow[\eps \to 0]{\mathrm{distribution}} \sqrt{\kappa} \int_Y L^{h,\delta}[\varphi](x) W^\alpha(dx),
\label{eq:main3}
\end{equation}
and
\begin{equation}
\sqrt{\kappa} \int_Y L^{h,\delta}[\varphi](x) W^\alpha(dx) \xrightarrow[h \to 0]{\mathrm{distribution}} 
\sqrt{\kappa}\int_Y \G \varphi(x) u_0(x) W^\alpha(dx),
\label{eq:main4}
\end{equation}
\end{enumerate}
where $W^\alpha(dy)$ is formally defined to be $\widetilde{W}^\alpha(y) dy$ and $\widetilde{W}^\alpha(y)$ is a Gaussian random field with covariance function given by $\E\{\widetilde{W}^\alpha(x)\widetilde{W}^\alpha(y)\} = |x-y|^{-\alpha}$.
\end{theorem}
\begin{remark}\label{rem:continuous}
We refer the reader to \cite{Khoshnevisan} for theories of stochastic integrals with respect to multi-parameter random processes. In fact, the limits above can be written as the following Gaussian distributions:
\begin{eqnarray}
\sigma \int_Y \G \varphi(x) u_0(x) dW(x) &&\overset{\mathrm{distribution}}{=}\mathcal{N}(0,\sigma^2 \|u_0 \G \varphi\|^2_{L^2} ),\label{eq:continuous11}\\
\sqrt{\kappa} \int_Y \G\varphi(x) u_0(x)  W^\alpha(dx) &&\overset{\mathrm{distribution}}{=} \mathcal{N}(0,\int_{Y^2} \frac{\kappa (u_0 \G \varphi) \otimes (u_0 \G \varphi)}{|x-y|^\alpha} dx dy).
\label{eq:continuous21}
\end{eqnarray}
\end{remark}
Comparing these results with Theorem \ref{thm:continuous} below which recalls the theory of random fluctuations in the continuous setting, and with the paths in Fig.~\ref{fig:cdPDE},  we find in the LRC setting that the multi-scale scheme \eqref{eq:hmsdef} captures the theoretical Gaussian limit fluctuations after $\eps$ and $h$ are successively sent to zero. Furthermore, the scheme is robust in the sense that it provides the correct fluctuations for arbitrary small patches with $0<\delta<h$ (both being independent of and hence much larger than $\eps$). For SRC medium, however, the correct limit for the random fluctuations is captured only  when $\delta=h$, that is $K_\delta = K$ for all $K \in \mathcal{T}_h$. The amplification effect in the case of $\delta<h$ is again characterized by $(h/\delta)^{\frac d 2}$. The main results hence generalize the findings of \cite{BJ-MMS-11} to a higher dimensional setting.
\begin{remark}
The main results are stated under the assumptions in Remark \ref{rem:hod}. When the ratios $\{|K|/|K_\delta|\}$ are not uniform over $\mathcal{T}_h$, the limit in \eqref{eq:main2} does not have a simple form and must account for the non-uniform amplification factors over different triangulation elements. Nevertheless, the main conclusions in the above result are not modified. This remark applies to the ODE setting in \cite{BJ-MMS-11} also.
\end{remark}

The rest of this paper is devoted to the proof of the main theorem. Preliminary material on random fields and the corrector theory in the continuous scale are provided in Section \ref{sec:continuous}. Then main ingredient of the proof is a conservative structure of the stiffness matrix associated to the multi-scale scheme; this is considered in section \ref{sec:analysis}. Similar structures have been observed and explored in other settings \cite{HWC-MC-99,BJ-MMS-11}. It allows us to write the discrete corrector in the form of oscillatory random integrals. Their limiting distributions are then characterized using well established techniques in \cite{FOP-SIAP-82,B-CLH-08,BGGJ-12}. This is done in Section \ref{sec:proof}. These sections also include some useful results on the scheme, such as the $H^1$ estimate of the solution to \eqref{eq:hmsdef}, which are interesting in their own right. We conclude this introduction by several comments.

\subsection{Further Discussions}
\label{sec:fdiss}

This paper studies the specific multi-scale scheme \eqref{eq:hmsdef} for the elliptic equation \eqref{eq:rpde} with a random potential. The analysis takes advantage of the conservative structure of the stiffness matrix. We refer to Proposition \ref{prop:stiff} below for a detailed statement. Other schemes possessing this property can be analyzed similarly. To simplify the presentation, we considered first-order nodal basis on a uniform triangulation. For higher order schemes in which basis functions occupy larger sub-domain of $Y$, and for general regular triangulation where different nodal basis may occupy different number of triangles, the structure in the stiffness matrix is more complicated. Nevertheless, we believe that  the analysis should extend without major differences to this more general setting. 

The scheme \eqref{eq:hmsdef} fits within the framework of HMM, which is a general methodology for designing multi-scale methods by exploiting scale separation and other special features of the problem. We refer to \cite{EMZ-JAMS-05} for references on this method applied to the following $\mathcal{L}_\eps$-problem:
\begin{equation*}
\left\{
\begin{aligned}
\mathcal{L}_\eps u_\eps(x,\omega) = \sum_{\alpha, \beta = 1}^d \frac{\partial}{\partial x_\alpha} \left(a_{\alpha\beta}(\frac{x}{\eps},\omega) \frac{\partial}{\partial x_\beta} u_\eps(x,\omega)\right) = f, &\quad& x \in Y,\\
u_\eps(x,\omega) = 0, &\quad& x \in \partial Y.
\end{aligned}
\right.
\end{equation*}
This problem is the higher dimensional version of \eqref{eq:rode}. Like the treatment there, the macro-solver is a conventional FEM on the triangulation $\mathcal{T}_h$ as for the homogenized equation. The missing effective stiffness matrix is approximated by solving a fine-scale problem on $K_\delta$. The problem \eqref{eq:rpde} considered in this paper is much easier. Indeed, the homogenized coefficient of \eqref{eq:rpde} is simply an average of $q^\eps$, whereas that of the $\mathcal{L}_\eps$-problem involves some auxiliary problem and is highly non-trivial; see \cite{Kozlov,PV-81}. In particular, the missing part of the macroscopic effective stiffness matrix for \eqref{eq:rpde} is just the integral of the zeroth order term, i.e.~$\sum_{K\in \mathcal{T}_h} |K| q_0 u^h v^h(x_K)$, say when barycenter numerical quadrature is used for the integrals. In the scheme \eqref{eq:hmsdef}, this missing datum $q_0 u^h v^h(x_K)$ are supplied by averaging $q^\eps u^h v^h$ around the barycenter $x_K$. Consequently, in the scheme of this paper, the macro-solver is the standard $P_1$ FEM on $\mathcal{T}_h$ and the micro-level computation is simply a fine-scale average on $K_\delta$. Though this scheme is very simple, our results show that it captures the homogenization and corrector effectively.

The amplification effect of the HMM scheme \eqref{eq:hmsdef} with $\delta < h$ in the SRC setting can be remedied as follows: On a typical triangle element $K \in \mathcal{T}_h$, instead of using one patch $K_\delta$, one may cover $K$ by a number of patches $\{K^i_\delta~|~ i=1,\cdots, (\frac{h}{\delta})^2\}$ for certain $\delta$ and average $q^\eps u^h v^h$ on these patches in parallel, and then combine them to approximate the effective integral of $q_0 u^h v^h$. Essentially this recovers the scheme \eqref{eq:hmsdef} with $\delta = h$ and hence rectifies the amplification of fluctuations. This technique has already been exploited in \cite{BJ-MMS-11} for the HMM scheme of \eqref{eq:rode}. 

Other multi-scale schemes and methodologies have been developed for the $\mathcal{L}_\eps$-problem using properties of the medium such as separation of scales, periodicity, or ergodicity, e.g. \cite{Babu75,Babu77,HWC-MC-99}. For instance, the {\it Multiscale Finite Element Method} (MsFEM) in \cite{HWC-MC-99} constructs oscillatory bases by solving $\mathcal{L}_\eps$-problems on the supports of the nodal bases $\{\phi^{ij}\}$ and uses the so-called over-sampling strategy to diminish the resonance errors introduced by the artificial boundary conditions of the local $\mathcal{L}_\eps$-problems. It would be interesting to investigate how random fluctuation are captured by this scheme and in particular what is the effect of the over-sampling strategy. The differential operator in \eqref{eq:hpde} does not exhibit such resonances, and hence this paper does not address such issues.  

Other multi-scale schemes approach differential operators with rough coefficients like $\mathcal{L}_\eps$ without assuming any separation of scales or special properties of the coefficient $a_{\alpha \beta}$. For instance, \cite{Owhadi11} constructs oscillatory bases by solving $\mathcal{L}_\eps$-problems on sub-domains that are larger than the supports of $\{\phi^{ij}\}$ but still small compared to the whole domain $Y$. It was proved there, using the so-called transfer property of the divergence operator \cite{Owhadi10}, that the resulting finite dimensional space can be used to solve the whole $\mathcal{L}_\eps$-problem with errors that are independent of the regularity of $\{a_{\alpha\beta}\}$. Analyzing the fluctuations in such schemes is beyond the scope of this paper. 

\section{Review of Corrector Theory in the Continuous Scale}
\label{sec:continuous}

In this section, we review the corrector theories for \eqref{eq:rpde} developed in \cite{FOP-SIAP-82,B-CLH-08}. They are formulated for the following random fields.

\subsection{Random field settings}
\label{sec:randf}
In the elliptic equation \eqref{eq:rpde}, the heterogeneous potential, denoted by $\tilde{q}_\eps(x)$ henceforth, consists of a slowly varying part $q_0(x)$ and a highly oscillating part $q_\eps(x)$. The latter is modeled as $q(\frac{x}{\eps},\omega)$, that is, spatially rescaled from some random field $q(x,\omega)$ defined on the probability space $(\Omega, \F, \Pb)$. In the sequel, $\E$ denotes the mathematical expectation with respect to the probability measure $\Pb$.

We assume that $q(x,\omega)$ is {\itshape stationary}. That is to say, for any positive integer $k$ and $k$-tuple $(x_1,\cdots,x_k)$, for any point $z$ and any Borel measurable set $\mathcal{A} \subset \R^k$, one has
\begin{equation*}
  \Pb\{(q(x_1),\cdots,q(x_k)) \in \mathcal{A} \} = \Pb\{(q(x_1 + 
z),\cdots,q(x_k + z)) \in \mathcal{A} \}.
\end{equation*}
With this assumption, $q$ admits an (auto-)correlation function $R(x)$ defined by
\begin{equation}
  R(x) : = \E q(y)q(y+x) = \E q(0)q(x).
  \label{eq:Rdef}
\end{equation}
It is easy to check that $R$ is symmetric, that is $R(x) = R(-x)$ for all $x \in \R^d$. It holds also that $R$ is a function of positive type in the sense that the $N$-by-$N$ matrix formed by $\{R(x_i - x_j)\}_{i,j=1}^N$ for any positive integer $N$ and $N$-tuple $x_1,\cdots, x_N \in \R^d$ is a non-negative definite matrix. Due to Bochner's theorem \cite{RSII}, the Fourier transform of $R$ is a positive Radon measure. In particular, when $R$ is integrable, one can define 
\begin{equation}
  \sigma^2 : = \int_{\R^d} R(x) dx,
  \label{eq:sigdef}
\end{equation}
and it is a finite non-negative number. Without loss of generality, we also assume that $q$ is mean-zero.

A key parameter of the random field that will determine different limiting correctors is the de-correlation rate. It is an indicator of how fast (with respect to distance) the random field becomes independent.

Recall that a random field $q(x,\omega)$ is said to be $\rho$-mixing 
with mixing coefficient $\rho$ if there exists some function $\rho(r)$, 
which maps $\R_+$ to $\R_+$ and vanishes as $r$ tends to infinity, so 
that for any Borel sets $A, B \subset \R^d$, the sub-$\sigma$-algebras 
$\F_{A}$ and $\F_{B}$ generated by the process restricted on $A$ and $B$ 
respectively de-correlate rapidly as follows:
\begin{equation}
  \sup_{\xi \in L^2(\F_A), \eta \in L^2(\F_B)} \left\vert \frac{\E\ \xi 
  \eta - \E \xi \ \E \eta}{(\mathrm{Var}\ \xi \ \mathrm{Var}\ \eta)^{1/2}} 
  \right\vert \le \rho( d(A, B)).
  \label{eq:alphadef}
\end{equation}
Here $d(A,B)$ is the distance between the sets $A$ and $B$. The function 
$\rho$ characterizes the decay of the dependence of the random 
field at different places. We refer the reader to \cite{Doukhan} for more 
information on mixing properties of random fields.

We consider two settings of random fields. In the first case, we say that $q(x,\omega)$ is {\itshape short range correlated} (SRC). This means 
\begin{itemize}
\item[(S1)] $q$ is $\rho$-mixing with mixing coefficient $\rho(r)$ such that $\rho(|x|) \in L^1(\R^d)$.
\item[(S2)] $|q(x)| \le C$ so that $\tilde{q}_\eps(x)$ is positive for a.e. $\omega \in \Omega$.
\item[(S3)] In this case, the correlation function $R(x)$ is integrable over $\R^d$ and we assume that $\sigma$ defined in \eqref{eq:sigdef} does not vanish, that is to say $\sigma > 0$.
\end{itemize}

In the second case, we say that $q(x,\omega)$ is {\itshape long range correlated} (LRC). In fact, we consider the very specific setting as follows.
\begin{enumerate}
\item[(L1)] $q(x)$ has the form $\Phi\circ g(x)$, where $\Phi: \R \to \R$ is function on the real line and $g(x,\omega)$ is a centered stationary Gaussian random field with unit variance and heavy tail, i.e.
\begin{equation}
R_g (x) : = \E \{g(y)g(y+x)\} \sim \kappa_g |x|^{-\alpha} \quad \text{ as 
} |x| \to \infty,
\label{eq:Rgdef}
\end{equation}
for some positive constant $\kappa_g$ and some real number $\alpha < d$.
\item[(L2)] The function $\Phi$ is uniformly bounded so that $\tilde{q}_\eps(x)$ is positive for a.e. $\omega$. Further, we assume the Fourier transform $\hat{\Phi}$ satisfies that $\int_{\R} |\hat{\Phi}|(1+|\xi|^3)$ is finite. 
\item[(L3)] We assume also that $\Phi$ has Hermite rank one, that is
\begin{equation}
\int_{\R} \Phi(s) e^{-\frac{s^2}{2}} ds = 0, \quad
V_1 := \quad \int_{\R} 
s\Phi(s) e^{-\frac{s^2}{2}} ds \ne 0.
\label{eq:Hrone}
\end{equation}
As a consequence $\kappa : = V_1^2 \kappa_g$ defines a positive number. For more information on the Hermite rank, we refer the reader to \cite{Taqqu79}.
\end{enumerate}

\subsection{Corrector theory in the continuous scale}

The corrector theory for the elliptic equation with random potential, that is the limiting distribution of the difference between $u_\eps$ and $u_0$ which solve \eqref{eq:rpde} and \eqref{eq:hpde} respectively, has been investigated in \cite{FOP-SIAP-82,B-CLH-08} in the SRC setting, and in \cite{BGGJ-12} in the LRC setting. Using the notations and random field settings introduced above, the results in dimension two of these references can be summarized as follows. 

\begin{theorem}[\cite{FOP-SIAP-82,B-CLH-08,BGGJ-12}] \label{thm:continuous} Let $u_\eps$ and $u_0$ be as above and let the dimension $d=2$. Denote by $G(x,y)$ be the fundamental solution to the Dirichlet problem \eqref{eq:hpde}. When the random potential $q(x,\omega)$ satisfies the SRC setting, we have
\begin{equation}
\frac{u_\eps(x) - u_0(x)}{\sqrt{\eps^d}} \xrightarrow[\eps\to 0]{\mathrm{distribution}} \sigma \int_Y G(x,y) u_0(y) dW(y)
\label{eq:continuous1}
\end{equation}
weakly in the spatial variable. When the random potential satisfies the LRC setting, we have
\begin{equation}
\frac{u_\eps(x) - u_0(x)}{\sqrt{\eps^\alpha}} \xrightarrow[\eps\to 0]{\mathrm{distribution}} \sqrt{\kappa} \int_Y G(x,y) u_0(y) W^\alpha(dy)
\label{eq:continuous2}
\end{equation}
weakly in the spatial variable.
\end{theorem}
Here, $W$ and $W^\alpha$ are the same as in Theorem \ref{thm:main}. The convergences above are weakly in the spatial variable in the sense of \eqref{eq:continuous11} and \eqref{eq:continuous21}.

\section{Analysis of the Discrete Equation}
\label{sec:analysis}

In this section, we analyze the heterogeneous multi-scale scheme \eqref{eq:hmsdef} in detail. In particular, we prove that the scheme with $\eps \ll \delta \le h$ admits a unique solution in the space $V^h$ that approximates $u_0$ in $H^1$. With the standard uniform triangulation, we show that the stiffness matrix associated to the scheme has some conservative form, which allows us to write the discrete corrector conveniently in terms of their coordinates. In the next section, we use this discrete representation to prove the main theorem.

\subsection{Well-posedness of the scheme}

The multi-scale scheme \eqref{eq:hmsdef} with $\delta = h$ coincides with the standard FEM and is well-posed. For the sake of completeness, we show that this holds also for $\delta < h$.

Recall that $V^h$ is the finite dimensional subspace of $H^1_0(Y)$ with nodal basis $\{\phi^{ij}\}$ defined in section \ref{sec:prepde}. We have defined three quadratic forms: $\A_\eps$ for the heterogeneous equation \eqref{eq:rpde}, $\A^{h,\delta}_\eps$ for the heterogeneous multi-scale scheme which is an approximation of $\A_\eps$ by local integration, and $\A^{h,\delta}_0$ which is like $\A^{h,\delta}_\eps$ but uses the mean coefficient $q_0$ only and which is an approximation of the quadratic form associated to the homogenized equation \eqref{eq:hpde}, that is
\begin{equation}
\A_0(u,v) = \int_Y \nabla u\cdot \nabla v + q_0 uv \ dx, \quad u, v \in H^1_0(Y).
\label{eq:A0def}
\end{equation}
Let $K$ be an element in the triangulation $\mathcal{T}_h$, and let $x_K$ denote its barycenter. Then one may check that $\A^{h,\delta}_{\eps}(u^h,v^h)$ is a weighted sum of terms of the form 
\begin{equation*}
\hat{\A}^{h,\delta}_{\eps}(u^h,v^h)[x_K] = \aint_{K_\delta} \nabla u^h \cdot \nabla v^h + (q_0 + q_\eps) u^h v^h \ dx.
\end{equation*}
We define $\hat{\A}^{h,\delta}_0(u^h,v^h)[x_K]$ similarly. Hereafter, the integral symbol with a dash in the middle denotes the averaged integral.

The characterize the difference between the discrete bilinear forms associated to the random and homogenized equations, we define
\begin{equation}
e(\HMM) := \max_{\substack{\ \\K \in \mathcal{T}_h}} \ \sup_{P_1(K) \ni u^h, v^h \ne 0}\frac{|K| \lvert \hat{\A}^{h,\delta}_{\eps}(u^h,v^h)[x_K] - \hat{\A}^{h,\delta}_0(u^h,v^h)[x_K] \rvert}{\|u^h\|_{H^1(K)} \|v^h\|_{H^1(K)}}.
\label{eq:eHMM}
\end{equation}
With this notation we have the following theorem.

\begin{theorem}\label{thm:wellposed} Assume that $q_0$ is a nonnegative $\mathcal{C}^1(\overline{Y})$ and $q_\eps(x) + q_0$ is uniformly bounded and nonnegative; assume also that $f \in \mathcal{C}(\overline{Y})$. There exist unique solutions $u^{h,\delta}_\eps$ and $u^{h,\delta}_0$ in $V^h$ for the numerical schemes \eqref{eq:hmsdef} and \eqref{eq:uh0def}. Let $u_0$ solves \eqref{eq:hpde}. Let the parameters $h$ and $\delta$ in the numerical schemes be fixed with $1 \le h/\delta \le C$. Then we have
\begin{equation}
\|u^{h,\delta}_\eps - u_0\|_{H^1} \le C(h + e(\HMM)),\label{eq:H1est}\\
\end{equation}
The above estimates hold also if we replace $u^{h,\delta}_\eps$ by $u^{h,\delta}_0$ and delete the term $e(\HMM)$.
\end{theorem}
\begin{proof} Let $p$ be either $\eps$ or $0$. The existence and uniqueness follow from\begin{equation*}
\A^{h,\delta}_p (v^h, v^h) \ge C\|v^h\|_{H^1}^2, \quad \text{ for any $v^h \in V^h$.}
\end{equation*}
Indeed, because $\nabla v^h$ is constant on $K \in \mathcal{T}_h$ and $q_0 + q_\eps$ is non-negative, we have
\begin{equation*}
\A^{h,\delta}_p (v^h, v^h) \ge \sum_{K \in \mathcal{T}_h} |K| \aint_{K_\delta} \lvert \nabla v^h\rvert^2 dx = \sum_{K \in \mathcal{T}_h} \int_K \lvert \nabla v^h\rvert^2 dx = |v^h|_{H^1}^2 \ge C\|v^h\|_{H^1}^2.
\end{equation*}
Here and in the sequel, $|\cdot|_{H^1}$ and $|\cdot|_{W^{k,p}}$ are the standard semi-norms of the corresponding Sobolev spaces.

We apply the first Strang lemma (Theorem 4.1.1 of \cite{Ciarlet78}), and obtain
\begin{equation*}
\|u_0 - u^{h,\delta}_\eps\|_{H^1} \le C \inf_{v^h \in V^h} \Big(\|u_0 - v^h\|_{H^1} + \sup_{w^h \in V^h} \frac{\lvert \A^{h,\delta}_{\eps}(v^h,w^h) - \A_0(v^h,w^h)\rvert}{\|w^h\|_{H^1}}\Big).
\end{equation*}
Set $v^h = \varPi u_0$, the projection of $u_0$ to the space $V^h$. From classical interpolation result, e.g. Theorem 3.1.6 of \cite{Ciarlet78}, we have
\begin{equation*}
\|\varPi u_0 - u_0\|_{H^1} \le Ch\|u_0\|_{H^2}.
\end{equation*}
For any $w^h \in V^h$, we have
\begin{equation*}
\lvert \A^{h,\delta}_{\eps}(v^h,w^h) - \A_0(v^h,w^h)\rvert \le \lvert \A^{h,\delta}_{\eps}(v^h,w^h) - \A^{h,\delta}_0(v^h,w^h)\rvert + \lvert \A^{h,\delta}_0(v^h,w^h) - \A_0(v^h,w^h)\rvert.
\end{equation*}
For the first term, we have
\begin{align*}
\lvert \A^{h,\delta}_{\eps}(v^h,w^h) - \A^{h,\delta}_0(v^h,w^h)\rvert &\le \sum_{K \in \mathcal{T}_h} |K|  \lvert \hat{\A}^{h,\delta}_{\eps}(v^h,w^h)[x_K] - \hat{\A}^{h,\delta}_0(v^h,w^h)[x_K]\rvert\\
&\le e(\HMM) \sum_{K \in \mathcal{T}_h} \|v^h\|_{H^1(K)} \|w^h\|_{H^1(K)}\\
&\le e(\HMM) \|v^h\|_{H^1} \|w^h\|_{H^1}.
\end{align*}
In the equalities above, we used the definition of $e(\HMM)$ and Cauchy-Schwarz respectively. For the second term, we first observe that
\begin{equation*}
\begin{aligned}
\A^{h,\delta}_0(v^h,w^h) - \A_0(v^h,w^h) = \sum_{K \in \mathcal{T}_h} & \left\{ \frac{|K|}{|K_\delta|} \left(\int_{K_\delta} q_0 v^h w^h\ dx - |K_\delta| (q_0 v^h w^h)(x_K) \right)\right.\\
 & \left. - \left(\int_{K} q_0 v^h w^h\ dx - |K| (q_0 v^h w^h)(x_K) \right)\right\}.
\end{aligned}
\end{equation*}
The items in the sum can be recognized as errors of barycenter numerical approximation of integrals. Error estimate for such numerical quadrature is discussed in the next lemma and by \eqref{eq:quaderr} we have that $\lvert \A^{h,\delta}_0(v^h,w^h) - \A_0(v^h,w^h) \rvert$ is bounded by
\begin{equation*}
\begin{aligned}
&\sum_{K \in \mathcal{T}_h} C\|q_0\|_{C^1} \left\{ \frac{h^2}{\delta^2} \delta \|v^h\|_{H^1(K_\delta)} \|w^h\|_{L^2(K_\delta)} + h \|v^h\|_{H^1(K)} \|w^h\|_{L^2(K)} \right\}\\
\le & Ch\|q_0\|_{C^1} \sum_{K \in \mathcal{T}_h} \|v^h\|_{H^1(K)} \|w^h\|_{L^2(K)} \le Ch\|q_0\|_{C^1} \|v^h\|_{H^1} \|w^h\|_{H^1}.
\end{aligned}
\end{equation*}
Combining the above estimates, we find that
\begin{equation*}
\|u_0 - u^{h,\delta}_\eps\|_{H^1}  \le \|\varPi u_0 - u_0\|_{H^1} + (e(\HMM) + Ch)\|\varPi u_0\|_{H^1} \le C(h + e(\HMM)).
\end{equation*}
The constant depends on $\|q_0\|_{C^1}$, $\|u_0\|_{H^2}$ and some uniform bound of $h/\delta$ and hence is independent of $\eps, h$ or $\delta$.
\end{proof}

The following lemma concerns error estimate for barycenter numerical quadrature of product of two functions in $P_1(K)$, the space of linear polynomials on a triangular element $K$. It is stated in the simplest setting thought it can be generalized to regular element easily. This lemma is used in the proof of the previous theorem.

\begin{lemma}\label{lem:hK2K} Let $\hat{K}$ be an isosceles right triangle with unit side length. Let $K$ be the image of $\hat{K}$ under some linear transform $F(\hat{x}) = B\hat{x} + b \in \R^2$. Assume $q_0 \in W^{1,\infty}(K)$. Then for any $v, w \in P_1(K)$, we have
\begin{equation}
\left\lvert \int_{K} q_0(x) v(x) w(x) \ dx - |K|(q_0vw)(x_K) \right\rvert \le C\|B\| \|q_0\|_{W^{1,\infty}(K)} \|v\|_{H^1(K)} \|w\|_{L^2(K)}.
\label{eq:quaderr}
\end{equation}
Here, $x_K$ is the barycenter of $K$; $\|B\|$ is the matrix norm of $B$.
\end{lemma}
\begin{proof} We follow the steps in the proof of \cite[Theorem 4.1.4 ]{Ciarlet78}. Consider any $\psi \in W^{1,\infty}(K)$ so that $\hat{\psi} = \psi \circ F$ is in $W^{1,\infty}(\hat{K})$. Let $|E(\psi w)|$ denote the error of the barycenter quadrature for the integral $\int_K \psi w dx$. After change of variables,
\begin{equation*}
E(\psi w) = |\mathrm{det} (B)| \left(\int_{\hat{K}} \hat{\psi}(\hat{x}) \hat{w}(\hat{x}) \ d\hat{x} - |\hat{K}|(\hat{\psi}\hat{w})(\hat{x}_{\hat{K}})\right) = |\mathrm{det} (B)| \hat{E}(\hat{\psi}\hat{w}).
\end{equation*}
On the reference element $\hat{K}$, since all norms on $P_1(\hat{K})$ are equivalent, we have
\begin{equation*}
|\hat{E}(\hat{\psi}\hat{w})| \le \hat{C} \|\hat{\psi}\|_{L^\infty(\hat{K})} \|\hat{w}\|_{L^{\infty}(\hat{K})} \le \hat{C} \|\hat{\psi}\|_{W^{1,\infty}(\hat{K})} \|\hat{w}\|_{L^2(\hat{K})}.
\end{equation*}
We view $\hat{E}(\cdot\ \hat{w}): \hat{\psi} \mapsto \hat{E}(\hat{\psi} \hat{w})$ as a linear functional on $W^{1,\infty}(\hat{K})$. The above estimate shows that $\hat{E}(\cdot \ \hat{w})$ is continuous with norm less than $\hat{C}\|\hat{w}\|_{L^2(\hat{K})}$. We check also that $\hat{E}(\cdot \ \hat{w})$ vanishes on $P_0(\hat{K})$, the space of constant functions on $\hat{K}$. Therefore, due to Bramble-Hilbert lemma \cite[Theorem 4.1.3]{Ciarlet78}, there exists some $\hat{C}$ such that for all $\ \hat{\psi} \in {W^{1,\infty}(\hat{K})}$,
\begin{equation*}
|\hat{E}(\hat{\psi} \hat{w})| \le \hat{C} \|\hat{E}(\cdot \ \hat{w})\|_{\mathcal{L}(W^{1,\infty} (\hat{K}))} \lvert\hat{\psi}\rvert_{W^{1,\infty}(\hat{K})} \le  \hat{C}\|\hat{w}\|_{L^2(\hat{K})}  |\hat{\psi}|_{W^{1,\infty}(\hat{K})}.
\end{equation*}
Take $\hat{\psi} = \hat{q_0} \hat{v}$. We check that
\begin{equation*}
\begin{aligned}
|\hat{\psi}|_{W^{1,\infty}(\hat{K})} &\le |\hat{q}_0|_{W^{1,\infty}(\hat{K})}\|\hat{v}\|_{L^{\infty}(\hat{K})} + \|\hat{q}_0\|_{L^\infty(\hat{K})} |\hat{v}|_{W^{1,\infty}(\hat{K})} \\
&\le \hat{C} \left( |\hat{q}_0|_{W^{1,\infty}(\hat{K})} \|\hat{v}\|_{L^2(\hat{K})} + \|\hat{q}_0\|_{L^\infty(\hat{K})} |\hat{v}|_{H^1(\hat{K})}\right).
\end{aligned}
\end{equation*}
The last inequality holds because $\hat{v} \in P_1(\hat{K})$ and all norms on $P_1(\hat{K})$ are equivalent. Finally, recall the relations \cite[Theorem 3.1.2]{Ciarlet78} that for any integer $m\ge 0$, any $q \in [1,\infty]$, and for any $\phi \in W^{m,p}(K)$,
\begin{equation}
|\hat{\phi}|_{W^{m,q} (\hat{K})} \le C\|B\|^m |\mathrm{det} (B)|^{-\frac 1 q} |\phi|_{W^{m,q}(K)}.
\label{eq:Wtrans}
\end{equation}
Apply this inequality to control the terms $|\hat{q}_0|_{W^{1,\infty}(\hat{K})}$ and $|\hat{v}|_{H^1(\hat{K})}$. On the other hand, for any $\phi \in L^p(K)$, we have
\begin{equation}
\|\hat{v}\|_{L^p(\hat{K})} = |\det(B)|^{-\frac 1 p}\|v\|_{L^p(K)}.
\label{eq:L2trans}
\end{equation}
Use this equality to estimate the $L^2$ norms of $\hat{v}$ and $\hat{w}$. Finally, combining the above estimates, we obtain the desired inequality.
\end{proof}

For the heterogeneous multi-scale error, we have the following result. We do not intend to make these estimates sharp. Nevertheless, the following theorem shows that the error in \eqref{eq:H1est} is small if the correlation length $\eps$ is much smaller than the parameters $h$ and $\delta$ of the HMM scheme, say when $\eps/(\delta h^2) \ll 1$ in the SRC setting and when $\eps/(\delta h^{2d/\alpha}) \ll 1$ in the LRC setting.

\begin{theorem}\label{thm:ehmm} Let dimension $d=2$. Let $e(\HMM)$ be the multi-scale heterogeneous error defined in \eqref{eq:eHMM}, and $h,\delta>0$ with $\delta\le h$ be the fixed parameters of the scheme \eqref{eq:hmsdef}. Then for $\eps$ sufficiently small, we have the following estimate:
\begin{equation}
\E \ e(\HMM) \le \left\{
\begin{aligned}
&C\frac{1}{h^d}\left(\frac{\eps}{\delta}\right)^{\frac d 2},  \quad & \text{ in the SRC setting,}\\
&C\frac{1}{h^d}\left(\frac{\eps}{\delta}\right)^{\frac \alpha 2}, \quad & \text{ in the LRC setting.}
\end{aligned}
\right.
\end{equation}
The constants $C$ above does not depend on $h,\delta$ or $\eps$.
\end{theorem}
\begin{proof} In the definition \eqref{eq:eHMM}, if we replace the $H^1$ norm on the denominator by $L^2$ norm and define for each $K \in \mathcal{T}_h$
\begin{equation*}
e_K : = \sup_{v, w \in P_1(K)} e_K(v,w) \quad \text{and} \quad e_K(v,w) : = \frac{|K| \lvert \hat{\A}^{h,\delta}_{\eps}(u,v)[x_K] - \hat{\A}^{h,\delta}_0(u,v)[x_K] \rvert}{\|v\|_{L^2(K)} \|w\|_{L^2(K)}},
\end{equation*}
then we check that $e(\HMM) \le \sup_{K \in \mathcal{T}_h} e_K$. Therefore, it suffices to estimate $e_K$.

For any $K \in \mathcal{T}_h$, let $\{\phi_m, m= 1, 2, 3\}$ be the standard basis functions of $P_1(K)$. As described above \eqref{eq:hmsdef}, each of these basis functions is a linear polynomial on $K$ that has value $1$ at one vertex of $K$ and vanishes at the other two vertices. Any function $v \in P_1(K)$ is identified with its coordinate $V \in \R^3$, that is by $v = \sum_{m=1}^3 V_m \phi_m$. We claim that there exist constants $0<\hat{c}<\hat{C}$, which are independent of $h,\delta$ and $\eps$, such that
\begin{equation}
\hat{c} h^{\frac d 2} \|V\| \le \|v\|_{L^2(K)} \le \hat{C} h^{\frac d 2} \|V\|,
\label{eq:P1norm}
\end{equation}
where $\|V\|$ is the Euclidean norm of $V$. To see this, recall the linear transform $F: \hat{K} \to K$ in the proof of Lemma \ref{lem:hK2K}. As before, a function $v \in P_1(K)$ is related to $\hat{v} = v\circ F \in P_1(\hat{K})$. In particular, $\hat{v}$ and $v$ have the same coefficients with respect to the basis $\{\hat{\phi}_m\}$ and $\{\phi_m\}$ respectively. In the finite dimensional space $P_1(\hat{K})$, since all norms are equivalent, we have $\hat{c} \|V\| \le \|\hat{v}\|_{L^2(\hat{K})} \le \hat{C} \|V\|$ for some $\hat{c}, \hat{C}$. Thanks to \eqref{eq:L2trans}, we obtain \eqref{eq:P1norm}.

For arbitrarily fixed $K \in \mathcal{T}_h$, $v, w \in P_1(K)$ and $v, w \not\equiv 0$, identified with their coefficients $V, W$, we explicitly calculate the expression of $e_K(v,w)$ and get
\begin{equation*}
e_K(v,w) = \frac{|K|}{|K_\delta| \|v\|_{L^2(K)} \|w\|_{L^2(K)}} \left| \sum_{m,n=1}^3 V_j W_m \int_{K_\delta} q_\eps(x) \phi_m(x) \phi_n(x) dx \right|.
\end{equation*}
Let us define, with $\chi_A$ denoting the indicator function of a set $A \subset \R^2$,
\begin{equation*}
X^\eps_{m,n} = \int q_\eps(x) \psi_{m,n}(x) dx, \quad \text{where} \quad \psi_{m,n}(x) = \chi_{K_\delta} (x) \phi_m (x) \phi_n (x).
\end{equation*}
Thanks to the Cauchy-Schwarz inequality and \eqref{eq:P1norm}, we have
\begin{equation}
e_K(v,w) \le \left(\frac{h}{\delta}\right)^d \frac{1}{\hat{c}^2 h^d} \left(\sum_{m,n=1}^3 |X^\eps_{m,n}|^2 \right)^{\frac 1 2},
\label{eq:eKbdd}
\end{equation}
where the ratio $|K|/|K_\delta|$ is replaced by $h^d\delta^{-d}$. Since this inequality is uniform in $v, w$, it is also satisfied by $e_K$.

To simplify the presentation, let $X^\eps$ and $\psi$ be the short-hand notation for $X^\eps_{m,n}$ and $\psi_{m,n}$ momentarily. Let us estimate $\E |X^\eps|^2$. We observe that $X^\eps$ is an integral of the highly oscillating random field $q_\eps$ against some slowly varying function $\psi$. Such integrals are studied carefully in \cite{B-CLH-08,BGGJ-12}. In the SRC setting, $\eps^{-\frac d 2} X^\eps$ converges in distribution to a mean-zero Gaussian variable with variance $\sigma^2 \|\psi\|_{L^2}^2$; see \cite[Theorem 3.8]{B-CLH-08}. In fact, its variance converges. Therefore, for sufficiently small $\eps$, we have
\begin{equation}
\E |X^\eps|^2 \le C \eps^d\|R\|_{L^1} \|\psi\|_{L^2}^2 = C \eps^d\|R\|_{L^1} \|\phi_m \phi_n \|_{L^2(K_\delta)}^2 \le C\|R\|_{L^1} \eps^d \delta^d.
\label{eq:estHMMS}
\end{equation}
Here $R$ is the correlation function of $q$ defined in \eqref{eq:Rdef}. We argued that $\|\phi_m \phi_n\|_{L^2(K_\delta)}^2 \le |K_\delta| \le C\delta^d$ because $|\phi_m \phi_n| \le 1$.

In the LRC setting, $\eps^{-\frac \alpha 2} X^\eps$ converges in distribution to a mean-zero Gaussian variable with variance $\|\psi\otimes \psi\|_{L^1(Y^2, \kappa|x-y|^{-\alpha} dxdy)}$; see \cite[Lemma 4.3]{BGGJ-12}. In fact, its variance converges. Consequently, for sufficiently small $\eps$, we have
\begin{equation}
\begin{aligned}
\E |X_\eps|^2 &\le C \eps^\alpha \iint_{K_\delta \times K_\delta} \frac{\kappa \psi(x)\psi(y)}{|x-y|^\alpha} dx dy \le C\eps^\alpha \|\psi\|_{L^{\frac{2d}{2d-\alpha}}}^2 \\
&=C\eps^\alpha \|\phi_m \phi_n \|_{L^{\frac{2d}{2d-\alpha}}}^2 \le C \eps^\alpha \delta^{2d-\alpha}.
\end{aligned}
\label{eq:estHMML}
\end{equation}
In the second inequality we used Hardy-Littlewood-Sobolev inequality \cite[Theorem 4.3]{LL-A}, and we calculated that $\|\phi_m \phi_n\|_{L^{\frac{2d}{2d-\alpha}}} \le |K_\delta|^{\frac{2d-\alpha}{2d}}$.

We observe that the above estimates of $\E|X^\eps|^2$ is uniform in $m,n$, and that the sum in \eqref{eq:eKbdd} has a finite number of terms independent of $h, \delta, \eps$. As a result, the inequalities \eqref{eq:estHMMS} and \eqref{eq:estHMML} show that $\E\ e_K$ is of order $({\frac \eps \delta})^{d/2}$ and $({\frac \eps \delta})^{\alpha/2}$ in the SRC and LRC settings respectively. Finally, we replace the maximum in \eqref{eq:eHMM} by the sum and get
\begin{equation}
\E \ e(\HMM) \le \sum_{K \in \mathcal{T}_h} \E \ e_K \le \frac{2}{h^d} \sup_{K \in \mathcal{T}_h} \E \ e_K.
\end{equation}
Here, $\frac{2}{h^d}$ is the number of elements in $\mathcal{T}_h$. Since the estimates \eqref{eq:estHMMS} and \eqref{eq:estHMML}  are uniform over $K \in \mathcal{T}_h$, we obtain the desired estimates.
\end{proof}

\subsection{Coordinate representation and conservative form}

The next step is to reformulate the multi-scale schemes \eqref{eq:hmsdef} and \eqref{eq:uh0def} as linear systems for the coordinates of the solutions in $V^h$, to investigate the structure of the associated stiffness matrices, and to write the discrete corrector $u^{h,\delta}_\eps - u^{h,\delta}_0$ in terms of their coordinates.

We start by introducing some useful notation. In the triangulation illustrated by Fig.~\ref{fig:tri}, we identify each grid point $(ih,jh)$ with a unique two dimensional index $(i,j)$. The set of inner grid points are denoted by $\Iinner=\{(i,j)~|~ 1\le i, j \le N-1\}$, and the set of all grid points including the boundary ones is denoted by $\Icomp = \{(i,j)~|~ 0\le i,j \le N\}$. We define six difference operators $\dind^{\pm}_s: \Iinner \to \Icomp$ as follows:
\begin{equation}
\dind^{\pm}_1 (i,j) = (i\pm 1, j), \quad \dind^{\pm}_2 (i,j) = (i, j\pm 1), \quad \dind^{\pm}_3 (i,j) = (i\pm 1, j\pm 1).
\end{equation}
Here, $s=1,2,3$ denotes three directions: horizontal, vertical and diagonal; the plus or minus sign indicates forward or backward differences.

In the sequel, we often write $(i,j)$ simply as $ij$. For each $ij \in \Iinner$, there corresponds a basis function $\phi^{ij}$ which is piecewise linear on each element $K \in \mathcal{T}_h$, has value one at $ij$ and has value zero at other nodal points. Any function $v^h$ in the space $V^h$ can be uniquely written as
$v^h(x) = \sum_{ij \in \Iinner} V_{ij} \phi^{ij}(x)$,
and the vector $(V_{ij}) \in \R^{(N-1)\times (N-1)}$ is called the coordinates of $v^h$. We identify $\R^{(N-1)\times (N-1)}$, the space for the coordinates, with $V^h$ itself. Now, the difference operators $\dind^{\pm}_s$ induce difference operators $D^{\pm}_s$ on $V^h$ as follows:
\begin{equation}
\begin{aligned}
D^+_s V_{ij} = V_{\dind^+_s ij} - V_{ij}, \quad D^-_s V_{ij} = V_{ij} - V_{\dind^-_s ij}.
\end{aligned}
\label{eq:Dpmdef}
\end{equation}
Note when $\dind^\pm_s ij$ lands outside of $\Iinner$, {\it i.e.}~on the boundary, the value $V_{\dind^\pm_s ij}$ is set to zero.

Using the coordinate representation of functions $u^{h,\delta}_\eps = \sum_{ij} U^\eps_{ij} \phi^{ij}$ and $u^{h,\delta}_0 = \sum_{ij} U^0_{ij} \phi^{ij}$, we can recast the heterogeneous multi-scale schemes \eqref{eq:hmsdef} and \eqref{eq:uh0def} as the following systems: for all $ij \in \Iinner$,
\begin{align}
\sum_{kl} A^\eps_{ijkl} U^\eps_{kl} &= \langle f, \phi^{ij} \rangle, \label{eq:Uepsdis}\\
\sum_{kl} A^0_{ijkl} U^0_{kl} &= \langle f, \phi^{ij} \rangle. \label{eq:U0dis}
\end{align}
Here, the stiffness matrices are defined by
\begin{equation*}
A^\eps_{ijkl} = \A^{h,\delta}_\eps(\phi^{ij},\phi^{kl}), \quad A^0_{ijkl} = \A^{h,\delta}_0(\phi^{ij},\phi^{kl}).
\end{equation*}
These stiffness matrices have the following structures.
\begin{proposition}\label{prop:stiff}
Let $A^p = (A^p_{ijkl})$ with $p = 0 \text{ or } \eps$ be the stiffness matrices above. We observe
\begin{enumerate}
\item[{\upshape (P1)}] $A^p_{ijkl} = A^p_{klij}$;

\item[{\upshape (P2)}] $A^p_{ijkl} = 0$ unless $kl \in \Iinner_{ij} := \{ij\} \bigcup \{\dind^{\pm}_s ij ~|~ s = 1, 2, 3\}$.

\item[{\upshape (P3)}] For any $ij \in \Iinner$, we have
\begin{equation}
A^p_{ijij} = d^p_{ij} - \sum_{s=1}^3 \Big( A^p_{ij \dind^+_s ij} + A^p_{ij \dind^-_s ij} \Big),
\label{eq:Apstructure}
\end{equation}
for some $d^p_{ij}$ that can be explicitly computed as in \eqref{eq:dpij} below.
\end{enumerate}
\end{proposition}

\begin{figure}
\caption{\itshape \label{fig:element} Left: The support of a basis function $\phi^{ij}$, denoted by $\K_{ij}$. Right: The shrunk integral region $\K_{ij}^\delta$.}
\setlength{\unitlength}{0.8cm}
\begin{center}
\begin{minipage}{0.4\textwidth}
\begin{picture}(7,6)
\multiput(2,1)(0.5,0){8}{\line(1,0){0.4}}
\multiput(2,3)(0.5,0){8}{\line(1,0){0.4}}
\multiput(2,5)(0.5,0){8}{\line(1,0){0.4}}
\multiput(2,1)(0,0.5){8}{\line(0,1){0.4}}
\multiput(4,1)(0,0.5){8}{\line(0,1){0.4}}
\multiput(6,1)(0,0.5){8}{\line(0,1){0.4}}
\multiput(2,1)(0.8,0.8){5}{\line(1,1){0.6}}
\put(4,0.3){$i$}
\put(1.5,3){$j$}
\put(2,1){\circle{0.2}}
\put(2,3){\circle{0.2}}
\put(4,1){\circle{0.2}}
\put(4,3){\circle*{0.2}}
\put(4,5){\circle{0.2}}
\put(6,3){\circle{0.2}}
\put(6,5){\circle{0.2}}
\thicklines
\multiput(2,3)(2,-2){2}{\line(1,1){2}}
\multiput(2,1)(2,4){2}{\line(1,0){2}}
\multiput(2,1)(4,2){2}{\line(0,1){2}}
\thinlines
\put(2.25,3.25){\line(1,0){3.75}}
\put(2.5,3.5){\line(1,0){3.5}}
\put(2.75,3.75){\line(1,0){3.25}}
\put(3,4){\line(1,0){3}}
\put(3.25,4.25){\line(1,0){2.75}}
\put(3.5,4.5){\line(1,0){2.5}}
\put(3.75,4.75){\line(1,0){2.25}}
\put(2,2.75){\line(1,0){3.75}}
\put(2,2.5){\line(1,0){3.5}}
\put(2,2.25){\line(1,0){3.25}}
\put(2,2){\line(1,0){3}}
\put(2,1.75){\line(1,0){2.75}}
\put(2,1.5){\line(1,0){2.5}}
\put(2,1.25){\line(1,0){2.25}}
\end{picture}
\end{minipage}
\begin{minipage}{0.4\textwidth}
\begin{picture}(7,6)
\multiput(2,1)(0.5,0){8}{\line(1,0){0.4}}
\multiput(2,3)(0.5,0){8}{\line(1,0){0.4}}
\multiput(2,5)(0.5,0){8}{\line(1,0){0.4}}
\multiput(2,1)(0,0.5){8}{\line(0,1){0.4}}
\multiput(4,1)(0,0.5){8}{\line(0,1){0.4}}
\multiput(6,1)(0,0.5){8}{\line(0,1){0.4}}
\multiput(2,1)(0.8,0.8){5}{\line(1,1){0.6}}
\put(4,0.3){$i$}
\put(1.5,3){$j$}
\multiput(2.67,2.33)(2,0){2}{\circle*{0.1}}
\multiput(3.33,3.67)(2,0){2}{\circle*{0.1}}
\put(3.33,1.66){\circle*{0.1}}
\put(4.67,4.33){\circle*{0.1}}
\thicklines
\multiput(2,3)(2,-2){2}{\line(1,1){2}}
\multiput(2,1)(2,4){2}{\line(1,0){2}}
\multiput(2,1)(4,2){2}{\line(0,1){2}}
\newsavebox{\shrinktri}
\savebox{\shrinktri}(0,0){
\thicklines
\put(-0.67,-0.33){\line(1,0){1}}
\put(0.33,-0.33){\line(0,1){1}}
\put(-0.67,-0.33){\line(1,1){1}}
}
\newsavebox{\shrinktris}
\savebox{\shrinktris}(0,0){
\thicklines
\put(-0.67,-0.33){\line(1,0){1}}
\put(0.33,-0.33){\line(0,1){1}}
\put(-0.67,-0.33){\line(1,1){1}}
\thinlines
\put(-0.33,0){\line(1,0){0.666}}
\put(0,0.33){\line(1,0){0.33}}
\put(0.33,-0.16){\line(-1,0){0.83}}
\put(0.33,0.16){\line(-1,0){0.49}}
\put(0.33,0.5){\line(-1,0){0.16}}
}
\multiput(3.333,1.8)(0,2){2}{\usebox{\shrinktri}}
\put(5.333,3.8){\usebox{\shrinktris}}
\newsavebox{\shrinkirt}
\savebox{\shrinkirt}(0,0){
\thicklines
\put(-0.33,0.33){\line(1,0){1}}
\put(-0.33,0.33){\line(0,-1){1}}
\put(-0.33,-0.67){\line(1,1){1}}
}
\multiput(2.666,2.2)(2,0){2}{\usebox{\shrinkirt}}
\put(4.666,4.2){\usebox{\shrinkirt}}
\end{picture}
\end{minipage}
\end{center}
\end{figure}
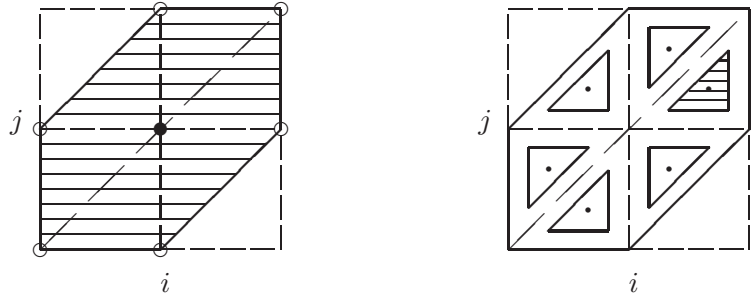
\begin{proof} The first two observations are obvious, so only the third one needs to be stressed. According to \eqref{eq:hmsdef} and \eqref{eq:uh0def}, to calculate $A^p_{ijij}$ we need to integrate the function
$|\nabla \phi^{ij}(x)|^2 + q^p(x) |\phi^{ij}(x)|^2$.
We observe that the support of $\phi^{ij}$, denoted by $\K_{ij}$, is a hexagon consisting of six triangle elements as illustrated in Fig.~\ref{fig:element}-Left. The integration is actually taken over $\K_{ij}^\delta$, the region obtained by shrinking the triangle elements in $\K_{ij}$ with respect to their barycenters as illustrated in Fig.~\ref{fig:element}-Right. Let us consider a typical triangle in $\K_{ij}$ with nodal points $ij, \dind^+_1 ij$ and $\dind^+_3 ij$. Abusing notation, we call it $K$ and the corresponding smaller triangle $K_\delta$. Note $K_\delta$ corresponds to the shaded region in the figure. On this region, the three non-zero basis functions are $\phi^{ij}, \phi^{\dind^+_1 ij}$ and $\phi^{\dind^+_3 ij}$. They satisfy
\begin{equation*}
\phi^{ij} + \phi^{\dind^+_1 ij} + \phi^{\dind^+_3 ij} = 1, \quad \nabla \phi^{ij} + 
\nabla\phi^{\dind^+_1 ij} + \nabla \phi^{\dind^+_3 ij} = 0.
\end{equation*}
Multiply $\phi^{ij}$ on both sides of the first equation, and $\nabla 
\phi^{ij}$ on the second equation. We have
\begin{equation*}
(\phi^{ij})^2 = \phi^{ij} -(\phi^{\dind^+_1 ij} + \phi^{\dind^+_3 ij})\phi^{ij}, \quad 
|\nabla \phi^{ij}|^2 = - (\nabla\phi^{\dind^+_1 ij} + \nabla \phi^{\dind^+_3 ij})\cdot 
\nabla \phi^{ij}.
\end{equation*}
Consequently, we have
\begin{equation*}
\begin{aligned}
\hat{\A}^{h,\delta}_p(\phi^{ij},\phi^{ij})[x_K] &= \aint_{K_\delta} |\nabla \phi^{ij}|^2 + q^p|\phi^{ij}|^2 dx\\
&= \aint_{K_\delta} q^p \phi^{ij} dx - \sum_{s=1,3} \aint_{K_\delta} \nabla \phi^{ij} \cdot \nabla \phi^{\dind^+_s ij} + q^p \phi^{ij} \phi^{\dind^+_s ij} dx\\
&= \aint_{K_\delta} q^p \phi^{ij} dx - \hat{\A}^{h,\delta}_p(\phi^{ij},\phi^{\dind^+_1 ij})[x_K] - \hat{\A}^{h,\delta}_p(\phi^{ij},\phi^{\dind^+_s ij})[x_K].
\end{aligned}
\end{equation*}
Summing over the integrals on all six triangles, and using the notations of $A^p$, $\A^{h,\delta}_p$ and $\hat{\A}^{h,\delta}_p$, $p=0, \eps$, we see that \eqref{eq:Apstructure} holds with
$d^p_{ij}$ defined by
\begin{equation}
d^p_{ij} = \sum_{K \in \K_{ij}} |K|\aint_{K_\delta} q^p \phi^{ij}  dx.
\label{eq:dpij}
\end{equation}
This completes the proof.
\end{proof}

It follows immediately that the matrix $A^p$ acts on vectors in $V^h$ as follows:
\begin{equation*}
(A^p V)_{ij} = \sum_{s = 1}^3 D^+_s (\alpha^{s,p}_{ij} D^-_s 
V_{ij}) + d^p_{ij} V_{ij},
\end{equation*}
where $\alpha^{s,p}_{ij}$ is short-hand notation for $A^p_{ij \dind^-_s ij}$ and it has the expression
\begin{equation*}
\alpha^{s,p}_{ij} := \sum_{K \in \K_{ij}} |K| \aint_{K_\delta} \nabla \phi^{ij} \cdot \nabla \phi^{\dind^-_s ij} + q^p \phi^{ij} \phi^{\dind^-_s ij} dx.
\end{equation*}
Note that when $\dind^\pm_s ij$ lands outside of $\Iinner$, {\it i.e.}~on the boundary, $\phi^{\dind^\pm_s ij}$ is the unique continuous function which is linear on each $K \in \mathcal{T}_h$, has value one at $\dind^\pm_s ij$ and value zero at all other nodal points. Finally, taking the difference of $A^\eps$ and $A^0$ we obtain
\begin{equation}
(A^\eps V - A^0 V)_{ij} = \sum_{s = 1}^3 D^+_s (\alpha^{s}_{\eps ij} D^-_s 
V_{ij}) + d_{\eps ij} V_{ij},
\label{eq:Aact}
\end{equation}
where the vectors $(\alpha^s_{\eps ij})$ and $(d_{\eps ij})$ are
\begin{align}
\alpha^s_{\eps ij} &:= \alpha^{s,\eps}_{ij} - \alpha^{s,0}_{ij} = \sum_{K \in \K_{ij}} |K| \aint_{K_\delta} q_\eps \phi^{ij} \phi^{\dind^-_s ij} dx, \label{eq:alphaeps}\\
d_{\eps ij} &:= d^\eps_{ij} - d^0_{ij} = \sum_{K \in \K_{ij}} |K|\aint_{K_\delta} q_\eps \phi^{ij}  dx. \label{eq:deps}
\end{align}

Formula \eqref{eq:Aact} is essential in our analysis because it provides an explicit expression of the discrete corrector $u^{h,\delta}_\eps - u^{h,\delta}_0$. Identify these solutions with the vectors $(U^{\eps}_{ij})$ and $(U^0_{ij})$ in (\ref{eq:Uepsdis}-\ref{eq:U0dis}). We verify that for all $ij \in \Iinner$,
\begin{equation*}
\sum_{kl} A^0_{ijkl}(U^\eps - U^0)_{kl} = - \sum_{kl} (A^\eps - 
A^0)_{ijkl} U^\eps_{kl}.
\end{equation*}
Let $G = (G_{ijkl})$ be the inverse of $A^0$. Solving the equation above, we get
\begin{equation}
(U^\eps - U^0)_{ij} = - \sum_{kl} G_{ijkl} [(A^\eps-A^0) U^\eps]_{kl}.
\label{eq:diffdis}
\end{equation}
Using the formula \eqref{eq:Aact} and summation by parts, we obtain
\begin{equation*}
\begin{aligned}
(U^\eps - U^0)_{ij} &= - \sum_{kl} G_{ijkl} \sum_{s=1}^3 D^+_s(\alpha^s_{\eps kl} D^-_s U^\eps_{kl}) + \sum_{kl} G_{ijkl} d_{\eps kl} U^\eps_{kl}\\
&= \sum_{kl} \sum_{s=1}^3 (D^-_s G_{ijkl}) (\alpha^s_{\eps kl} D^-_s U^\eps_{kl}) - \sum_{kl} G_{ijkl} d_{\eps kl} U^\eps_{kl}.
\end{aligned}
\end{equation*}
Here and in the sequel, $D^\pm_s$ acts on $G$ as defined in \eqref{eq:Dpmdef} but in the second pair of indices, namely $kl$ here. We can also write this expression as
\begin{align}
(U^\eps - U^0)_{ij} &= \sum_{kl} \sum_{s=1}^3 (D^-_s G_{ijkl}) (\alpha^s_{\eps kl} D^-_s U^0_{kl}) - \sum_{kl} G_{ijkl} d_{\eps kl} U^0_{kl}
\label{eq:decomp}\\
&+ \sum_{kl} \sum_{s=1}^3 (D^-_s G_{ijkl}) [\alpha^s_{\eps kl} D^-_s (U^\eps - U^0)_{kl}] - \sum_{kl} G_{ijkl} d_{\eps kl} (U^\eps - U^0)_{kl}\notag.
\end{align}
This decomposition formula will be the starting point of our analysis in the next section.

\section{Proof of the Main Results}
\label{sec:proof}

In this section, we prove Theorem \ref{thm:main} using the coordinate representation \eqref{eq:decomp} of the discrete corrector.

We briefly describe the strategy of proof. We first show that $\|U^\eps - U^0\|_{\ell^2}$ is small in mean square when $\eps$ goes to zero while $h$ and $\delta$ are fixed (Lemma \ref{lem:UepsMU0}). This indicates that the first line in the representation \eqref{eq:decomp}, i.e.~the terms that are linear in $\alpha^s_\eps$ and $d_\eps$, is dominant while the second line is asymptotically small (Lemma \ref{lem:rem}). Then to prove the main theorem, using the coordinate representation \eqref{eq:decomp}, we write the normalized corrector integrated with a test function, more precisely its dominant part, as an integral of the highly oscillating random field $q_\eps(x)$ with certain slowly varying function, and invoke the aforementioned theorems in \cite{B-CLH-08,BGGJ-12} to prove the convergence in distribution as $\eps \downarrow 0$ while $h$ and $\delta$ are fixed. Finally, the limit as $h, \delta \downarrow 0$ afterwards with the ratio $\frac{h}{\delta} \ge 1$ fixed boils down to convergence of Gaussian random variables, and the proof is somewhat standard.

\begin{lemma}\label{lem:UepsMU0}
Let $U^\eps_{ij}$ and $U^0_{ij}$ be the coordinates of the 
solutions to the random and the deterministic discrete equations \eqref{eq:hmsdef} 
and \eqref{eq:uh0def} respectively.  Suppose that there exist some 
constants $C>0$ and $\gamma_j \in \R, j= 1,\cdots, 4$, which are possibly negative, so that
\begin{equation}
|D^-_s G_{ijkl}| \le Ch^{\gamma_1}, |D^-_s U^\eps_{ij}| \le 
Ch^{\gamma_2}, |G_{ijkl}| \le Ch^{\gamma_3} \text{and } |U^\eps_{ij}| 
\le Ch^{\gamma_4}
\label{eq:GUDGU}
\end{equation}
for any $s = 1,2,3$ and any indices $ij, kl \in \Iinner$. Let $d=2$. Then the following holds.

\vskip 1em 
{\upshape(1)} If the random process $q$ satisfies the SRC setting, we have
\begin{equation}
\E \|U^\eps - U^0\|_{\ell^2}^2 \le Ch^{2 (\min\{\gamma_1 + \gamma_2, 
\gamma_3 + \gamma_4\})-d} \|R\|_1 \left(\frac{\eps}{\delta}\right)^d.
\label{eq:L2cs}
\end{equation}

{\upshape(2)} If the random process $q$ satisfies the LRC setting, we have
\begin{equation}
\E \|U^\eps - U^0\|_{\ell^2}^2 \le C(\alpha,\kappa)h^{2(\min\{\gamma_1 + \gamma_2, 
\gamma_3 + \gamma_4\})-d} \left(\frac{\eps}{\delta}\right)^\alpha.
\label{eq:L2cl}
\end{equation}
The constant $C$ does not depend on $h, \delta$ or $\eps$. 
\end{lemma}
\begin{remark} The assumption \eqref{eq:GUDGU} is not a restriction because $\gamma_j$ there can be chosen negative. Indeed, consider a typical triangle $K \in \mathcal{T}_h$, namely the one with vertices $(ij,i-1j,ij+1)$, and a function $v \in P_1(K)$; the $L^2(K)$ norm of $v$ is related to its coordinate by \eqref{eq:P1norm}. Similarly, the $W^{1,q}(K)$ semi-norm of $v$ is related to its coordinates by
\begin{equation}
|v|_{W^{1,q}(K)} = Ch^{{\frac 2 q}-1} \|(D^-_1 V_{ij}, D^-_2 V_{ij+1})\| = Ch^{{\frac 2 q} - 1}\big(|D^-_1 V_{ij}|^2 + |D^-_2 V_{ij+1}|^2\big)^{\frac 1 2}.
\label{eq:DVW}
\end{equation}
This follows from the fact that $\nabla v^h\lvert_K$ is a constant vector $(D^-_1 V_{ij}, D^-_2 V_{ij+1})/h$.

Now for $u^{h,\delta}_\eps$, we know its $H^1$ norm is bounded independent of $h$ and $\eps$. Applying the results above we find that $|U^\eps_{ij}| \le Ch^{-1}$ and $|D^-_1 U^\eps_{ij}| \le C$. Other coordinates of $U^\eps$ and $D^-_s U^\eps$ can be estimated in the same way. Hence, we may choose $\gamma_2 = 0$ and $\gamma_4 = -1$. Similarly, the discrete Green's function $G^h(x,y) = \sum_{ij,kl} G_{ijkl}\phi^{ij}(x)\phi^{kl}(y)$ is known to have $W^{1,q}$ norm for some $q < 2$ bounded by $C|\log h|$ for any fixed $x$; see \cite[Theorem 5.1]{Goldstein80}. Using \eqref{eq:P1norm} and \eqref{eq:DVW} we may choose $\gamma_1$ and $\gamma_3$ properly, say $\gamma_1 = \gamma_3 = -1$. $\Box$
\end{remark}

\begin{proof}[Proof of Lemma \ref{lem:UepsMU0}]
Apply the bounds in \eqref{eq:GUDGU} to the representation of $(U^\eps - U^0)_{ij}$ above \eqref{eq:decomp}, and then take expectation and use Cauchy-Schwarz. We get
\begin{equation}
\E \lvert U^\eps - U^0\rvert_{ij}^2 \le C h^{-d} h^{2(\gamma_1+\gamma_2)} \sum_{s=1}^3 \sum_{kl} \E |\alpha^s_{\eps kl}|^2  + C h^{-d} h^{2(\gamma_3+\gamma_4)} \sum_{kl} \E |d_{\eps kl}|^2.
\label{eq:Udiffij}
\end{equation}
Here, $h^{-d}$ is the number of nodal points (up to a factor of $d$), i.e.~$|\Iinner|$. It suffices to estimate $\E |\alpha^s_{\eps kl}|^2$ and $\E |d_{\eps kl}|^2$. We rewrite \eqref{eq:alphaeps} and \eqref{eq:deps} as
\begin{equation}
\alpha^s_{\eps kl} = \int q_\eps(x) a^s_{kl} (x) dx, \quad \quad d_{\eps kl} = \int q_\eps(x) b_{kl} (x) dx,
\label{eq:alphadredef}
\end{equation}
with $a^s_{\eps}$ and $b_\eps$ defined by
\begin{equation}
a^s_{kl}(x) = \sum_{K \in \K_{kl}} \frac{|K|}{|K_\delta|} \chi_{K_\delta} (x) \phi^{kl}(x) \phi^{\dind^-_s kl}(x), \quad b_{kl}(x) = \sum_{K \in \K_{kl}} \frac{|K|}{|K_\delta|} \chi_{K_\delta} (x) \phi^{kl}(x).
\label{eq:abdef}
\end{equation}
Above, $\K_{kl}$ is defined below \eqref{eq:Apstructure}. We check that $|K|/|K_\delta| = (h/\delta)^d$ and that $a^s_{kl}$ and $b_{kl}$ are uniformly bounded on $Y$. Hence, $\alpha^s_{\eps kl}$ and $d_{\eps kl}$ can be recognized as oscillatory integrals of the highly oscillatory random field $q_\eps(x)$ against some slowly varying functions. Such integrals are well understood. In fact, $\alpha^s_\eps$ has the same form as $X^\eps$ in the proof of Theorem \ref{thm:ehmm} and can be estimated in the same manner. In the SRC setting, we have that
\begin{equation}
\E |\alpha^s_{\eps kl}|^2 \le C \eps^d \|R\|_{L^1} \|a^s_{kl}\|_{L^2}^2 \le C \|R\|_{L^1} h^{2d} ({\frac \eps \delta})^d.
\label{eq:bepsbdd}
\end{equation}
In the LRC setting, the above estimate should be replaced by
\begin{equation}
\E |\alpha^s_{\eps kl}|^2 \le C\eps^\alpha \|a^s_{kl}\otimes a^s_{kl}\|_{L^1(Y\times Y, |x-y|^{-\alpha} dxdy)} \le C(\alpha,\kappa) h^{2d} ({\frac \eps \delta})^\alpha.
\label{eq:bepsbddl}
\end{equation}
The mean square of $d_{\eps kl}$ can be similarly estimated. Substitute these estimates into \eqref{eq:Udiffij} to control the mean square of $(U^\eps - U^0)_{ij}$; note that the sum over $kl$ introduces a factor of $h^{-d}$ which is the number of items in the sum. The estimates of $(U^\eps - U^0)_{ij}$ are uniform in $ij$, summation over $ij$ yields the desired results. Note that this additional summation introduces another $h^{-d}$ to the estimates.
\end{proof}

\begin{lemma}\label{lem:rem} Under the same conditions of the previous lemma, we have
\begin{equation}
(U^\eps - U^0)_{ij} = \sum_{kl \in \Iinner} \sum_{s=1}^3 (D^-_s G_{ijkl}) \alpha^s_{\eps kl} (D^-_s U^0_{kl}) - \sum_{kl \in \Iinner} G_{ijkl} d_{\eps kl} U^0_{kl} + r^\eps_{ij}.
\label{eq:decrij}
\end{equation}
Further, the error term $r^\eps_{ij}$ satisfies
\begin{equation}
\sup_{ij \in \Iinner} \E |r^\eps_{ij}| \le
\left\{
\begin{aligned}
&C h^{\min\{\gamma_1,\gamma_3\} + \min\{\gamma_1+\gamma_2, \gamma_3 + \gamma_4\}}\left(\frac{\eps}{\delta}\right)^d, \quad & \text{ in the SRC setting},\\
&C h^{\min\{\gamma_1,\gamma_3\} 
+ \min\{\gamma_1+\gamma_2, \gamma_3 + \gamma_4\}} \left(\frac{\eps}{\delta}\right)^\alpha, \quad &\text{ in the LRC setting},
\end{aligned}
\right.
\label{eq:rijbdd}
\end{equation}
where $\gamma_1,\cdots, \gamma_4$ are as in \eqref{eq:GUDGU} and can be negative.
\end{lemma}
\begin{proof} The decomposition holds with
\begin{equation}
r^\eps_{ij} = \sum_{s=1}^3 \sum_{kl \in \Iinner} (D^-_s G_{ijkl}) \alpha^{s}_{\eps kl} D^-_s (U^\eps_{kl} - U^0)_{kl} - \sum_{kl \in \Iinner} G_{ijkl} d_{\eps kl} (U^\eps - U^0)_{kl}.
\label{eq:rijdef}
\end{equation}
Bound the $D^-_s G_{ijkl}$ and $G_{ijkl}$ terms by \eqref{eq:GUDGU}, and use Cauchy-Schwarz. We get
\begin{equation*}
|r^\eps_{ij}| \le Ch^{\gamma_1} \sum_{s=1}^3 \|\alpha^s_\eps\|_{\ell^2} \|D^-_s (U^\eps - U^0)\|_{\ell^2} + Ch^{\gamma_3} \|d_\eps\|_{\ell^2} \|U^\eps - U^0\|_{\ell^2}.
\end{equation*}
Note that $\|D^-_s (U^\eps - U^0)\|_{\ell^2}^2 \le C\|U^\eps - U^0\|_{\ell^2}$. Take expectation and use Cauchy-Schwarz again to get
\begin{equation}
\E |r^\eps_{ij}| \le Ch^{\gamma_1} \sum_{s=1}^3 \left(\E \|\alpha^s_\eps\|_{\ell^2}^2 \E \|U^\eps - U^0\|_{\ell^2}^2\right)^{\frac 1 2} + Ch^{\gamma_3} \left(\E \|d_\eps\|_{\ell^2}^2  \E \|U^\eps - U^0\|_{\ell^2}^2 \right)^{\frac 1 2}.
\label{eq:repsijbdd}
\end{equation}
Summing over $kl$ in the estimates \eqref{eq:bepsbdd} and \eqref{eq:bepsbddl}, we have
\begin{equation*}
\E \|\alpha^s_\eps\|_{\ell^2}^2 \le
\left\{
\begin{aligned}
&C \|R\|_{L^1} h^d ({\frac \eps \delta})^d, \quad & \text{ in the SRC setting},\\
&C(\alpha, \kappa) h^d ({\frac \eps \delta})^\alpha,  \quad &\text{ in the LRC setting}.
\end{aligned}
\right.
\end{equation*}
The same estimates hold also for $\E\|d_\eps\|_{\ell^2}^2$. Substituting these estimates, together with \eqref{eq:L2cs} and \eqref{eq:L2cl}, into \eqref{eq:repsijbdd} completes the proof.
\end{proof}

Now we prove the main theorem of the paper. Let $\G^{h,\delta}$ denote the solution operator of the discrete equation \eqref{eq:uh0def} which corresponds to the homogenized equation \eqref{eq:hpde}. Using the coordinate representation, the solution to \eqref{eq:uh0def} is then
\begin{equation}
\G^{h,\delta} f(x) = \sum_{ij \in \Iinner} \Big(\sum_{kl \in \Iinner} G_{ijkl} \langle f, \phi^{kl}\rangle\Big) \phi^{ij}(x).
\label{eq:Ghddef}
\end{equation}

\begin{proof}[Proof of Theorem \ref{thm:main}] Take any test function $\varphi \in \mathcal{C}^2(\overline{Y})$. Let $m^h$ denote the function $\G^{h,\delta} \varphi$. Its coordinate vector $(M_{ij})$ is then $M_{ij} = \sum_{kl} G_{ijkl} \langle \varphi, \phi^{kl}\rangle$ thanks to \eqref{eq:Ghddef}. Let $\beta = d$ in the SRC setting and $\beta = \alpha$ in the LRC setting. By \eqref{eq:decrij}, we have
\begin{align}
&\frac{1}{\sqrt{\eps^\beta}}  \int_Y \varphi(x)[u^{h,\delta}_\eps-u^{h,\delta}_0] dx = \frac{1}{\sqrt{\eps^\beta}} \sum_{ij} (U^\eps - U^0)_{ij} \langle \varphi, \phi^{ij} \rangle \notag\\
=& \frac{1}{\sqrt{\eps^\beta}} \sum_{ij} \left(\sum_{kl} \sum_{s=1}^3 (D^-_s G_{ijkl}) \alpha^s_{\eps kl} (D^-_s U^0_{kl}) - \sum_{kl} G_{ijkl} d_{\eps kl} U^0_{kl} + r^\eps_{ij} \right) \langle \varphi, \phi^{ij} \rangle \notag \\
=& \frac{1}{\sqrt{\eps^\beta}} \left[\sum_{kl} \sum_{s=1}^3 (D^-_s M_{kl}) \alpha^s_{\eps kl} (D^-_s U^0_{kl}) - \sum_{kl} M_{kl} d_{\eps kl} U^0_{kl} \right] + \frac{1}{\sqrt{\eps^\beta}} \sum_{ij} r^\eps_{ij} \langle \varphi, \phi^{ij}\rangle. \label{eq:udiffphi}
\end{align}
In the last equality, we used the fact that $G_{ijkl} = G_{klij}$ and recognized the coordinate $M_{kl}$.

{\itshape First convergence as $\eps \to $ 0 while $h$ and $\delta$ are fixed.} Let us control the last term in \eqref{eq:udiffphi}. Thanks to the estimate \eqref{eq:rijbdd}, we have
\begin{equation}
\E \left\lvert \frac{1}{\sqrt{\eps^\beta}} \sum_{ij \in \Iinner} r^\eps_{ij} \langle \varphi, \phi^{ij}\rangle \right\rvert \le \frac{1}{\sqrt{\eps^\beta}} \sup_{ij \in \Iinner} \big(\E \lvert r^\eps_{ij} \rvert \big) \sum_{ij} \lvert \langle \varphi, \phi^{ij}\rangle\rvert \le C(h,\delta) \|\varphi\|_{L^1} \sqrt{\eps^\beta}.
\label{eq:IIIbdd}
\end{equation}
Above $C(h,\delta)$ is a constant, say some negative powers of $h$ and $\delta$. As $\eps$ goes to zero while $h$ and $\delta$ are fixed, the term above converges to zero in $L^1(\Pb)$ and does not contribute to the limiting distribution of \eqref{eq:udiffphi}. The other terms there are linear in $(\alpha^s_{\eps kl})$ and $(d_{\eps kl})$. By \eqref{eq:abdef}, we find that
\begin{equation}
\begin{aligned}
\frac{1}{\sqrt{\eps^\beta}}  \int_Y \varphi(x)[u^{h,\delta}_\eps-u^{h,\delta}_0] dx &\simeq \frac{1}{\sqrt{\eps^\beta}} \int_Y q_\eps(x) L^{h,\delta}_1(x) dx + \frac{1}{\sqrt{\eps^\beta}} \int_Y q_\eps(x) L^{h,\delta}_2(x) dx\\
&= \frac{1}{\sqrt{\eps^\beta}} \int_Y q_\eps(x) L^{h,\delta}(x) dx.
\end{aligned}
\label{eq:wcorRep}
\end{equation}
Here, $L^{h,\delta}_j$, $j=1,2$ and $L^{h,\delta} = L^{h,\delta}_1 + L^{h,\delta}_2$ depend on $\varphi$ through $M$ and are defined by
\begin{equation}
\begin{aligned}
&L^{h,\delta}_1(x) = \sum_{kl} \sum_{s=1}^3 (D^-_s M_{kl}) (D^-_s U^0_{kl}) a^s_{kl}(x)\\
=\ &\sum_{kl} \sum_{K \in \K_{kl}} \frac{|K|}{|K_\delta|} \chi_{K_\delta} (x) \sum_{s=1}^3 (D^-_s M_{kl})(D^-_s U^0_{kl}) \phi^{kl}(x) \phi^{\dind^-_s kl}(x),\label{eq:Lhdef}\\
&L^{h,\delta}_2(x) = -\sum_{kl} b_{kl}(x) M_{kl} U^0_{kl} = -\sum_{kl} \sum_{K \in \K_{kl}} \frac{|K|}{|K_\delta|} \chi_{K_\delta} (x) M_{kl} U^0_{kl} \phi^{kl}(x)\\
=\ &-\sum_{K \in \mathcal{T}_h} \frac{|K|}{|K_\delta|}  \chi_{K_\delta} (x) \sum_{kl \in \Iinner_K} M_{kl} U^0_{kl} \phi^{kl}(x) = -\sum_{K \in \mathcal{T}_h} \frac{|K|}{|K_\delta|}  \chi_{K_\delta} (x) \varPi^h (m^h u^{h,\delta}_0) (x).
\end{aligned}
\end{equation}
Here, $\Iinner_K = \{kl \in \Iinner ~|~ (kh,lh) \in \overline{K}\}$ and $\varPi^h (m^h u^{h,\delta}_0)$ is the projection in $V^h$ of the function $m^h u^{h,\delta}_0$. Now the convergence results \eqref{eq:main1} and \eqref{eq:main3} of Theorem \ref{thm:main} follow from the representation \eqref{eq:wcorRep} and the aforementioned results on limiting distribution of oscillatory integrals, namely Theorem 3.8 of \cite{B-CLH-08} for the SRC setting and Lemma 4.3 of \cite{BGGJ-12} for the LRC setting.

{\itshape Second convergence as $h, \delta \downarrow 0$ with $h/\delta \ge 1$ fixed, SRC setting}. Now we prove \eqref{eq:main2}. It concerns the limiting distribution, as $h$ goes to zero, of the Gaussian random variable which is obtained as the limiting distribution in the first step.

We have the following key observation:
\begin{equation}
L^{h,\delta}_1 \longrightarrow 0 \text{ in $L^\infty(Y)$ as $h \to 0$.}
\label{eq:Lh1to0}
\end{equation}
Indeed, for any fixed $x \in Y$, since $|\phi^{ij}| \le 1$ uniformly and $|K|/|K_\delta| = (h\delta^{-1})^2$, we have
\begin{equation*}
\lvert L^{h,\delta}_1(x) \rvert \le C \left(\frac{h}{\delta}\right)^2 h^2 \sum_{s=1}^3 \left\| \frac{D^-_s M_{kl}}{h}\right\|_{\ell^2} \left\| \frac{D^-_s U^0_{kl}}{h}\right\|_{\ell^2} \le C \left(\frac{h}{\delta}\right)^2 h^2 \lvert m^h \rvert_{H^1} \lvert u^{h,\delta}_0 \rvert_{H^1}.
\end{equation*}
Since $u^{h,\delta}_0$ and $m^h$ are yielded form the scheme \eqref{eq:uh0def} for smooth right hand side $f$ and $\varphi$, they have bounded $H^1$ norms. We assume that the ratio $h/\delta$ is fixed while $h$ is sent to zero. Therefore, the above estimate shows that $L^{h,\delta}_1$ goes to zero uniformly, proving the claim.

According to \eqref{eq:wcorRep}, the left hand side of \eqref{eq:main2} can be written as
\begin{equation}
\sigma \int_Y L^{h,\delta}_1(x) dW(x) + \sigma \int_Y L^{h,\delta}_2(x) dW(x).
\label{eq:2Cdec}
\end{equation}
To prove \eqref{eq:main2}, it suffices to show that the second term above converges to the right hand side of \eqref{eq:main2} while the first term above converges in probability to zero. Since all random variables involved are Gaussian, we only need to calculate their variances. Thanks to It\^o's isometry, we have
\begin{equation*}
\Var \ \sigma \int_Y L^{h,\delta}_1(x) dW(x) = \sigma^2 \int_Y \lvert L^{h,\delta}_1(x) \rvert^2 dx.
\end{equation*}
Due to \eqref{eq:Lh1to0}, the above variance goes to zero, proving our claim for the first term. For the second one, we have again
\begin{equation*}
\Var \ \sigma \int_Y L^{h,\delta}_2(x) dW(x) = \sigma^2 \int_Y \lvert L^{h,\delta}_2(x) \rvert^2 dx = \left(\frac{\sigma h}{\delta}\right)^2 \sum_{K \in \mathcal{T}_h}  |K|\aint_{K_\delta} \lvert \varPi^h (m^h u^{h,\delta}_0) (x) \rvert^2 dx.
\end{equation*}
We recognize the sum in the last term as a barycenter approximation of the integral that gives the $L^2$ norm square of $\varPi^h (m^h u^{h,\delta}_0)$. Thanks to Lemma \ref{lem:hmscov} below, $\|\varPi^h (m^h u^{h,\delta}_0)\|_{L^2}$ converges to $\|u_0 \G \varphi\|_{L^2}$ by applying \eqref{eq:muh2mu0} with $f_1 = \varphi$ and $f_2 = f$. This implies that the variance of the second term in \eqref{eq:2Cdec} converges to $(\sigma h/\delta)^2 \|u_0 \G \varphi\|_{L^2}^2$, proving \eqref{eq:main2}.

{\itshape Second convergence as $h, \delta \to 0$ with $h/\delta \ge 1$ fixed, LRC setting.} Now we prove \eqref{eq:main4}. Like in \eqref{eq:2Cdec}, we can write the left hand side of \eqref{eq:main4} as a sum of two Gaussian random variables. Using a modified isometry, we write the variance of the first variable as
\begin{equation*}
\Var \ \sigma \int_Y L^{h,\delta}_1(x) W^\alpha(dx) = \iint_{Y^2}  \frac{\kappa L^{h,\delta}_1(x) L^{h,\delta}_1(y)}{|x-y|^\alpha}dxdy = \mathscr{I}(L^{h,\delta}_1).
\end{equation*}
Here, we define the operator $\mathscr{I}: L^{\frac 4 {4-\alpha}} \to \R$ as
\begin{equation}
\mathscr{I}(g) : = \|g\otimes g\|_{L^1(Y^2, \kappa |x-y|^{-\alpha} dx dy)} = \iint_{Y^2} \frac{\kappa g(x) g(y)}{|x-y|^{\alpha}} dx dy.
\label{eq:SIdef}
\end{equation}
Recalling the Hardy-Littlewood-Sobolev inequality, Theorem 4.3 of \cite{LL-A}, we have
\begin{equation}
\lvert \mathscr{I}(g) \rvert \le \kappa C(\alpha) \|g\|^2_{L^{\frac 4 {4-\alpha}}}.
\label{eq:SIcont}
\end{equation}
Due to \eqref{eq:Lh1to0}, the term in the equation above \eqref{eq:SIdef} converges to zero and doesn't contribute to the limiting distribution. For the the term with $L^{h,\delta}_2$, we have
\begin{equation*}
\begin{aligned}
\Var \ \sigma \int_Y L^{h,\delta}_2(x) W^\alpha(dx) &= \iint_{Y^2}  \frac{\kappa L^{h,\delta}_2(x) L^{h,\delta}_2(y)}{|x-y|^\alpha}dxdy\\
&= \sum_{K \in \mathcal{T}_h} \sum_{K' \in \mathcal{T}_h} |K|^2 \aint_{K_\delta} \aint_{K'_\delta} \frac{\kappa \varPi^h (m^h u^{h,\delta}_0)(x) \varPi^h (m^h u^{h,\delta}_0)(y)}{|x-y|^\alpha} dx dy.
\end{aligned}
\end{equation*}
We recognize the last sum as the barycenter approximation of $\mathscr{I}(\varPi^h (m^h u^{h,\delta}_0))$. Now \eqref{eq:SIcont} shows that $\mathscr{I}$ is continuous on $L^{\frac 4 {4-\alpha}}$. Since $\alpha < 2$ and $\frac 4 {4-\alpha} < 2$, we have the inclusion $L^2(Y) \subset L^{\frac 4 {4-\alpha}}(Y)$. Therefore $\mathscr{I}$ is also continuous on $L^2(Y)$. Applying \eqref{eq:muh2mu0} with $f_1 = \varphi$ and $f_2 = f$, we conclude that $\mathscr{I}(\varPi^h (m^h u^{h,\delta}_0))$ converges to $\mathscr{I}(u_0 \G \varphi)$. This proves \eqref{eq:main4} and completes the proof of the theorem.
\end{proof}

It remains to prove the following key lemma concerning the convergence of product of solutions yielded from the discrete equation \eqref{eq:uh0def}.

\begin{lemma} \label{lem:hmscov} Let $\G^{h,\delta}$ be the Green's operator of the scheme \eqref{eq:uh0def}. For any two functions $f_j \in \mathcal{C}^2(\overline{Y})$, $j = 1, 2$, let $\varPi^h ( \G^{h,\delta} f_1 \G^{h,\delta} f_2)$ be the projection in $V^h$ of the product of $\G^{h,\delta} f_1$ and $\G^{h,\delta} f_2$. We have that
\begin{equation}
\varPi^h ( \G^{h,\delta} f_1 \G^{h,\delta} f_2) \xrightarrow{\ \ L^2\ } \G f_1 \G f_2, \quad \text{ as $h \to 0$ with $h/\delta \ge 1$ fixed.}
\label{eq:muh2mu0}
\end{equation}
As before, $\G$ above is the Green's operator of the homogenized equation \eqref{eq:hpde}.
\end{lemma}

\begin{proof} To simplify notation, let us denote the function $\G^{h,\delta} f_j$ by $\tilde{u}^h_j$, the functions $\G f_j$ by $u_j$, $j=1,2$.

The key to the proof relies on $L^\infty$ error estimates for finite element methods. Such results are classic for the scheme with $h=\delta$ as proved in \cite{Scott82,Scott-76}. For $\delta < h$, as explained before we may view the scheme as the standard finite element method with (barycenter) numerical quadrature for evaluation of integrations. $L^\infty$ error estimates for such practical schemes are more involved but were obtained in \cite{Wahlbin78, Goldstein80}. In particular, the piecewise linear FEM with numerical quadrature was considered in Theorem 5.1 of \cite{Goldstein80}, which shows
\begin{equation*}
\|\tilde{u}^h_j - u_j\|_{L^\infty} \le Ch^2|\log h| \|f_j\|_{W^{2,\infty}}.
\end{equation*}
Since $\tilde{u}^h_j$, $j=1,2$, are bounded, the above also implies that
\begin{equation}
\|\tilde{u}^h_1 \tilde{u}^h_2 - u_1 u_2\|_{L^\infty} \le Ch^2|\log h| \|f_j\|^2_{W^{2,\infty}}.
\label{eq:uhjLinf}
\end{equation}

In fact, Theorem 5.1 of \cite{Goldstein80} also shows that
\begin{equation*}
\|\tilde{u}^h_j\|_{W^{1,\infty}} \le \|u^h_j\|_{W^{1,\infty}} + Ch|\log h| (\|u_j\|_{W^{2,\infty}} + \|f_j\|_{W^{2,\infty}}).
\end{equation*}
Here, $u^h_j$ is the FEM solution with $h=\delta$. The above estimate shows that $\tilde{u}^h_j$ is in $W^{1}_\infty$. Since $u_j$ are bounded, we check that $\tilde{u}^h_1 \tilde{u}^h_2 \in W^{1}_\infty$. From classical interpolation estimates, e.g. taking $k=m=0$, $p=\infty$ and $q=2$ in Theorem 3.1.6 of \cite{Ciarlet78}, we have
\begin{equation*}
\|\tilde{u}^h_1 \tilde{u}^h_2 - \varPi^h_K (\tilde{u}^h_1 \tilde{u}^h_2)\|_{L^2(K)} \le C|K|^{\frac 1 2} h |\tilde{u}^h_1 \tilde{u}^h_2|_{W^{1,\infty}(Y)}.
\end{equation*}
Here, $\varPi^h_K$ is the projection on the triangle element $K$. Summing over $K \in \mathcal{T}_h$, we have
\begin{equation}
\|\tilde{u}^h_1 \tilde{u}^h_2 - \varPi^h (\tilde{u}^h_1 \tilde{u}^h_2)\|_{L^2(Y)} \le C h \|\tilde{u}^h_1 \tilde{u}^h_2\|_{W^{1,\infty}}.
\label{eq:uhjint}
\end{equation}
Note that \eqref{eq:uhjLinf} controls $\|\tilde{u}^h_1 \tilde{u}^h_2\|_{W^{1,\infty}}$. Sending $h$ to zero, we finish the proof.
\end{proof}
\section*{Acknowledgments} The authors would like to thank the reviewers for a thorough reading of the manuscript and remarks that helped with the presentation of the results.

\end{document}